\documentclass[12pt]{article}

\usepackage
{amsmath,amsfonts,latexsym,graphicx,amssymb,amsthm}

\usepackage{color}
\usepackage{authblk}
\usepackage[new]{old-arrows}

\newcommand{\C}{\mathbb{C}}
\newcommand{\F}{\mathbb{F}}
\newcommand{\G}{\mathbb{G}}

\newcommand{\Q}{\mathbb{Q}}
\newcommand{\R}{\mathbb{R}}
\newcommand{\Z}{\mathbb{Z}}

\newcommand{\Aut}{\mathrm{Aut}}
\newcommand{\Card}{\mathrm{Card\,}}
\newcommand{\Cl}{\mathrm{Cl}}
\newcommand{\Coker}{\mathrm{Coker}}
\newcommand{\cp}{\mathrm{cp}}
\newcommand{\Der}{\mathrm{Der}}

\newcommand{\End}{\mathrm{End}}

\newcommand{\fcp}{\mathrm{fcp}}
\newcommand{\Gal}{{\mathrm {Gal}}}

\newcommand{\Hom}{{\mathrm {Hom}}}
\newcommand{\he}{{\mathrm{ht}}}

\newcommand{\Image}{\mathrm{Im}}

\newcommand{\Ker}{\mathrm{Ker}}
\newcommand{\lcm}{\mathrm{lcm}}
\newcommand{\Lie}{\mathrm{Lie}}
\newcommand{\Max}{\mathrm{Max}}
\newcommand{\Min}{\mathrm{Min}}

\newcommand{\N}{\mathrm{N}}

\newcommand{\rk}{\mathrm{rk}}
\newcommand{\Sol}{{\mathrm {Sol}}\,}
\newcommand{\Spec}{{\mathrm {Spec}}\,}

\renewcommand{\mod}{\,\mathrm{mod}\,}

\newtheorem{thm}{Theorem}

\newtheorem*{Main1}{Theorem \ref{Main1}}
\newtheorem*{Main2}{Theorem \ref{Main2}}
\newtheorem*{Main3}{Theorem \ref{Main3}}

\newtheorem*{corAnosov}{Corollary \ref{Anosov}}
\newtheorem*{corAffine}{Corollary \ref{Affine}}
\newtheorem*{corNonaffine}{Corollary \ref{Nonaffine}}

\newtheorem*{Benoistthm}{Benoist's Theorem}
\newtheorem*{Malcevthm}{Malcev's Theorem}
\newtheorem*{Malcevcor}{Malcev's Corollary}
\newtheorem*{Manningthm}{Manning's Theorem}

\newtheorem{prop}[thm]{Proposition}
\newtheorem{lemma}{Lemma}
\newtheorem{cor}[thm]{Corollary}

\newcommand{\fa}{\mathfrak{a}}

\newcommand{\fm}{\mathfrak{m}}
\newcommand{\fn}{\mathfrak{n}}

\newcommand{\fz}{\mathfrak{z}}

\begin{document}
\title{Which Nilpotent Groups are Self-Similar?
\footnote{Research supported 
by CNRS-UMR 5028
and Labex MILYON/ANR-10-LABX-0070.}}

\author [$\dag$]{Olivier Mathieu
\footnote{Institut Camille Jordan, 
Universit\'e de Lyon. Email:
mathieu@math.univ-lyon1.fr}}

\maketitle

\begin{abstract} 
Let $\Gamma$ be a finitely generated torsion free nilpotent group, and let $A^\omega$ be the  space of infinite words over 
a finite alphabet $A$. We investigate two types of
self-similar actions of $\Gamma$ on $A^\omega$, namely the
faithfull actions with  dense orbits and the free actions.
A criterion for the existence  of a self-similar action of each type is established.

Two corollaries about the nilmanifolds are deduced.
The first involves the nilmanifolds 
endowed with an Anosov diffeomorphism, and  the
second about the existence of an affine structure.

Then we investigate the virtual actions of $\Gamma$, i.e. actions of a subgroup $\Gamma'$ of finite index. A formula, with some number theoretical content, is found for 
the  minimal cardinal of an alphabet $A$ endowed
with a virtual self-similar action on $A^\omega$ of each type.
\end{abstract}

\author{{\bf Mathematics Subject Classification}
37B10-20G30-53C30}


\vskip1cm
\section*{Introduction}

\noindent {\it 1. General introduction}

\noindent Let $A$ be a finite alphabet and 
let $A^\omega$ be the topological space of infinite words
$a_1a_2\dots$ over $A$, where the topology  of
$A^\omega=\varprojlim A^n$ is the pro-finite topology.

An {\it action} of a group $\Gamma$ on  $A^\omega$  is called 
{\it self-similar} iff for any  $\gamma\in\Gamma$ and $a\in A$  
there exists $\gamma_a\in\Gamma$ and  $b\in A$  such that

\centerline{$\gamma(aw)=b\gamma_a(w)$ for any $w\in A^\omega$.}

The group $\Gamma$ is called {\it self-similar} 
(respectively {\it densely self-similar},
respectively {\it  freely self-similar},
respectively {\it  freely densely self-similar})
if $\Gamma$ admits  a faithfull self-similar action
(respectively a faithfull  self-similar
action with dense orbits, respectively a free self-similar
action, respectively a free self-similar
action with dense orbits) 
on $A^\omega$ for some finite alphabet $A$.

Self-similar groups appeared in the early eighties,  
in the works
of R. Grigorchuk \cite{GR80} \cite{GR83} and in the joint works of
N. Gupta and S. Sidki \cite{GS83} \cite{GS84}. See also
the monography  \cite{N} for
an extensive account before 2005 and \cite{NS} \cite{BS} \cite{FKS} \cite{KS} \cite{GNZ} for more recent works.
A general question is 

\centerline{\it which  groups $\Gamma$ are 
(merely, or densely ...) self-similar?} 

This paper brings an answer for finitely generated torsion-free nilpotent groups $\Gamma$, called FGTF nilpotent groups in the sequel. 
Then we will connect the main result with topics involving differential geometry and arithmetic groups.

The  systematic study of self-similar actions of nilpotent groups started with \cite{BS1}, and the
previous question has been raised  in some talks of S. Sidki. 

\bigskip
\noindent {\it 2. The main results}

\noindent
A few definitions are now required.
A {\it grading} of a Lie algebra $\fm$ is a decomposition
$\fm=\oplus_{n\in \Z}\, \fm_n$   such  that
$[\fm_n,\fm_m]\subset\fm_{n+m}$ for all $n,\,m\in\Z$. It is called {\it special} if $\fm_0\cap\fz=0$, where
$\fz$ is the center of $\fm$. It is called
{\it very special} if $\fm_0=0$.

Let's assume now that $\Gamma$ is a
FGTF nilpotent group. 
By Malcev Theory \cite{Ma}\cite{Ra}, $\Gamma$ is 
a cocompact lattice in a unique connected, simply connected 
(or CSC in what follows) nilpotent Lie group $N$.
Let $\fn^{\R}$ be the Lie algebra of $N$
and set $\fn^{\C}=\C\otimes_{\R}\fn^{\R}$.

The main results, proved in Section 7, are the following

\begin{Main1}

The group $\Gamma$ is densely self-similar iff
the Lie algebra $\fn^\C$ admits a special grading.
\end{Main1}

\begin{Main2} The following assertions are equivalent

(i) The group $\Gamma$ is freely self-similar,

(ii) the group $\Gamma$ is freely densely self-similar, and

(iii) the Lie algebra $\fn^\C$ admits a very
special grading.
\end{Main2}

As a consequence, let's mention

\begin{corAnosov} Let $M$ be a nilmanifold endowed with
an Anosov diffeomorphism. Then there a free self-similar action
of $\pi_1(M)$ with dense orbits on $A^\omega$, for some finite $A$.
\end{corAnosov}

\begin{corAffine} Let $M$ be a nilmanifold. If
$\pi_1(M)$ is freely  self-similar,
then $M$ is affine complete.
\end{corAffine}

Among FGTF nilpotent groups, some
of them   are self-similar but not densely 
self-similar.  Some of them  are  not even  self-similar, since  Theorem 2 implies the next  

\begin{corNonaffine} Let $M$ be one of the  non-affine  nilmanifolds appearing in \cite{Ben}. Then $\pi_1(M)$ is not self-similar.
\end{corNonaffine}

\smallskip
\noindent
{\it 3. A concrete version of Theorems 
\ref{Main1} and \ref{Main2}}

\noindent
Let $N$ be a CSC nilpotent  Lie group, with Lie algebra $\fn^{\R}$. Let's assume  that $N$ contains some  cocompact lattices $\Gamma$. 
By definition, the {\it degree} of a self-similar action of
$\Gamma$ on $A^\omega$ is $\Card\,A$. We ask the following question

\noindent\centerline{\it For a given cocompact lattice $\Gamma\subset N$, 
what is the minimal degree} 

\noindent\centerline {\it degree of a faithfull
(or a free)   self-similar action with dense orbits?}

More notions are now defined. 
Recall that the {\it commensurable class} $\xi$ of a cocompact lattice
$\Gamma_0\subset N$ is the set of all cocompact lattices of $N$ which 
share with $\Gamma_0$ a  subgroup of finite index.  The 
{\it complexity} $\cp\, \xi$  (respectively the free complexity
$\fcp\, \xi$) of the class $\xi$ is the minimal
degree of a   self-similar action of $\Gamma$
with dense orbits (respectively a free self-similar action
of $\Gamma$), for some $\Gamma\in\xi$.

For any algebraic number $\lambda\neq 0$, set
$d(\lambda)=\Card\,{\cal O}(\lambda)/\pi_{\lambda}$,
where ${\cal O}(\lambda)$ is the ring of integers of
$\Q(\lambda)$ and $\pi_{\lambda}=\{x\in {\cal O}(\lambda)\vert x\lambda\in {\cal O}(\lambda)\}$. For any isomorphism $h$ of a finite dimensional vector space over $\Q$, set 

\centerline{$\he\,h=\prod_{\lambda\in\Spec h/\Gal(\Q)}\,d(\lambda)^{m_\lambda}$,}

\noindent where  $\Spec h/\Gal(\Q)$ is the list of 
eigenvalues of $h$ up to conjugacy by
$\Gal(\Q)$ and where $m_\lambda$ is the multiplicity
of the eigenvalue $\lambda$.

By Malcev's Theory, the commensurable class $\xi$ determines a
canonical $\Q$-form $\fn(\xi)$ of the Lie algebra 
$\fn^{\R}$. 
Let  ${\cal S}(\fn(\xi))$ (respectively ${\cal V}(\fn(\xi))$)
be the set of all  $f\in\Aut\,\fn(\xi)$ such that  
$\Spec\,f\vert\fz^\C$ 
(respectively $\Spec\,f$) contains no algebraic integer.

\begin{Main3} We have

\centerline{$\cp\,\xi=\Min_{h\in{\cal S}(\fn(\xi))}\,\he\, h$, and}

\centerline{$\fcp\,\xi=\Min_{h\in{\cal V}(\fn(\xi))}\,\he\, h$.}
\end{Main3}

If, in the previous statement,  ${\cal S}(\fn(\xi))$ is empty, then
the equality $\cp\,\xi=\infty$  means that no $\Gamma\in \xi$ admits a faithfull self-similar action with dense $\Gamma$-orbits. 

Theorem \ref{Main3} answers the previous question only  for the commensurable classes $\xi$. For an individual $\Gamma\in\xi$,  it
provides  some ugly  estimates for the minimal
degree of $\Gamma$-actions, and nothing
better can be expected.

The framework of nonabelian Galois cohomology shows
the concreteness of Theorem \ref{Main3}. Up to conjugacy,
the commensurable classes in $N$ are classified by the $\Q$-forms of some classical objects with a prescribed
$\R$-form, see Corollary 4 of ch. 9, and their complexity is
an invariant of the arithmetic group $\Aut\,\fn(\xi)$.

As an illustration of the previous vague sentence,
we investigate a class 
${\cal N}$ of  CSC nilpotent Lie groups $N$, with Lie algebra
$\fn^\R$. The commensurable classes 
$\xi(q)$ in $N$  are
classified, up to conjugacy,  by the positive definite quadratic forms $q$ on $\Q^2$.  Then, we have

\centerline{$\cp\,\xi(q)=F(d)^{e(N)}$}

\noindent where $e(N)$ is an  invariant of the special grading of 
$\C\otimes \fn^{\R}$, where  $-d$ is the  discriminant of  $q$, and 
where $F(d)$ is the norm of a specific ideal in
$\Q(\sqrt{-d})$, see Theorem \ref{Main4} and Lemma \ref{F(d)1}. 

In particular, $N$ contains some commensurable classes of arbitrarily high
complexity. In a forthcoming paper \cite{Mat},  more complicated examples are investigated, but the formulas are less explicit.

\bigskip
\noindent {\it 4. About the proofs.}
The proofs of the paper are based on different ideas.  Theorem \ref{Main0}, which  is  a statement about rational points of algebraic tori,  is 
the key step in the proof of Theorems \ref{Main1}, \ref{Main2} and
\ref{Main4}. It is based on standard results of number theory, including the Cebotarev's Theorem. It is connected with the density of rational
points for connected groups proved by A. Borel \cite{Bor}, see also \cite{Ro}.

Also, the proof of  Corollary \ref{Anosov} is
based on a paper of A. Manning \cite{Man}
about Anosov diffeomorphisms. The proof of 
Corollary \ref{Nonaffine}
is based on very difficult computations, which, fortunately, were  entirely  done in \cite{Ben}.

\section{Self-similar actions and self-similar data}

Let $\Gamma$ be a group.
This section explains the correspondence between the faithfull transitive self-similar $\Gamma$-actions  and some virtual endomorphisms
of $\Gamma$, called {\it self-similar data}.
Usually self-similar actions are actions
on a rooted tree $A^*$, see \cite{N}. Here the groups
are acting on the boundary $A^\omega$ of $A^*$. This
equivalent viewpoint is better
adapted to our setting.

\bigskip\noindent
{\it 1.1 Transitive self-similar actions}

\noindent In addition of the 
definitions of the introduction, the
following technical notion of transitivity will be used.

A self-similar action of $\Gamma$ on
$A^\omega$ induces an action of $\Gamma$ on $A$.
Indeed, for $a,\,b\in A$ and $\gamma\in\Gamma$,
we have $\gamma(a)=b$ if

\centerline{$\gamma(aw)=b\gamma_a(w)$,}

\noindent for all $w\in A^\omega$. A self-similar action
is called {\it transitive} if the induced action on
$A$ is transitive. The group $\Gamma$
is called {\it transitive self-similar} if it admits
a faithfull transitive self-similar action.

Similarly the self-similar action of $\Gamma$ on
$A^\omega$ induces an action of $\Gamma$ on
each level set $A^n$. Such an action is often called 
{\it level transitive} if $\Gamma$ acts transitively
on each level $A^n$. Obviously, the level 
transitive actions on $A^*$ of \cite{N}
corresponds with the  actions
on $A^\omega$ with dense orbits.

\bigskip\noindent
{\it 1.2 Core and $f$-core}

\noindent
Let $\Gamma$ be a group and $\Gamma'$ be a subgroup.
The {\it core} of $\Gamma'$ is the biggest normal subgroup
$K\triangleleft G$ with $K\subset \Gamma'$. Equivalently 
the core is the kernel of the left action of $\Gamma$ on
$\Gamma/\Gamma'$. 

Now let $f:\Gamma'\to\Gamma$ be a group morphism. By defintion
the {\it $f$-core} is the biggest normal subgroup 
$K\triangleleft G$ with $K\subset \Gamma'$ and $f(K)\subset K$.

\bigskip\noindent
{\it 1.3 Self-similar data}

\noindent  Let $\Gamma$ be a group. A {\it virtual endomorphism}
of $\Gamma$ is a pair $(\Gamma',f)$, where $\Gamma'$ is a subgroup of finite index and $f:\Gamma'\to\Gamma$ is a group morphism. 
A {\it self-similar datum} 
is a virtual endomorphism  $(\Gamma',f)$  with a
trivial  $f$-core.

Assume given a faithfull transitive self-similar action 
of $\Gamma$  on  $A^\omega$. Let $a\in A$, and let
$\Gamma'$ be the stabilizer of $a$. By definition, for
each $\gamma\in\Gamma'$ there is a unique $\gamma_a\in\Gamma$ such that

\centerline{$\gamma (aw)=a\gamma_a(w)$,}

\noindent for any $w\in A^\omega$. 
Let $f:\Gamma'\rightarrow \Gamma$ be the map 
$\gamma\mapsto\gamma_a$. Since the action is faithfull, 
$\gamma_a$ is uniquely determined and  $f$ is a group morphism. 
Also it follows from  Proposition 2.7.4 and 2.7.5 of
\cite{N} that  the $f$-core of $\Gamma'$ is the kernel of
the action, therefore it is trivial.
Hence $(\Gamma',f)$ is a self-similar datum.

Conversely, a virtual endomorphism 
$(\Gamma',f)$ determines a  transitive self-similar action 
of $\Gamma$  on  $A^\omega$, where 
$A\simeq\Gamma/\Gamma'$. Moreover the $f$-core is
the kernel of the corresponding action, see
ch 2 of \cite{N} for details, especially subsection 2.5.5 of \cite{N}). In conclusion, we have

\begin{lemma}\label{corr1} Let $\Gamma$ be a group. There is a correspondence between the self-similar  data 
$(\Gamma',f)$ and 
the faithfull transitive self-similar actions of $\Gamma$ on 
 $A^\omega$, where $A\simeq \Gamma/\Gamma'$. 

\end{lemma}

\smallskip
\noindent 
This correspondence is indeed a bijection up to conjugacy,
see \cite{N} for a precise statement.

\bigskip\noindent
{\it 1.4 Good self-similar data}

\noindent Let $\Gamma$ be a group, and let
$(\Gamma',f)$ be a virtual endomorphism. Let $\Gamma_n$ be the subgroups of $\Gamma$ inductively defined by
$\Gamma_0=\Gamma$, $\Gamma_1=\Gamma'$ and for $n\geq 2$

\centerline{$\Gamma_{n}=\{\gamma\in\Gamma_{n-1}\vert\,f(\gamma)\in\Gamma_{n-1}\}$}

\begin{lemma}\label{good} The sequence $n\mapsto [\Gamma_n:\Gamma_{n+1}]$ is not increasing.
\end{lemma}

\begin{proof} For $n>0$, the  map $f$ induces an injection of the set $\Gamma_n/\Gamma_{n+1}$ into $\Gamma_{n-1}/\Gamma_{n}$, thus
we have $[\Gamma_n:\Gamma_{n+1}]\leq [\Gamma_{n-1}:\Gamma_{n}]$.
\end{proof}

The  virtual endomorphism $(\Gamma',f)$ is called {\it good}
if $[\Gamma_n:\Gamma_{n+1}]=[\Gamma/\Gamma']$ for all
$n$. 

Let $(\Gamma',f)$ be a  self-similar datum,
and let $A^*$ be the corresponding tree on which $\Gamma$ acts. If $a$ is the distinguished point in 
$A\simeq \Gamma/\Gamma'$, then $\Gamma_n$ is the stabilizer of $a^n$. If  the  self-similar datum $(\Gamma',f)$ is good, then $[\Gamma:\Gamma_n]=\Card A^n$ and therefore $\Gamma$ acts transitively on $A^n$. Exactly as before, we have

\begin{lemma}\label{corr2} Let $\Gamma$ be a group. There is a correspondence between the good self-similar  data $(\Gamma',f)$ and the
faithfull  self-similar actions of $\Gamma$ on 
$A^\omega$ with dense orbits, where $A\simeq \Gamma/\Gamma'$. 
\end{lemma} 

\bigskip\noindent
{\it 1.5 Fractal self-similar data}

\noindent
Let $\Gamma$ be a group. A {\it self-similar datum} 
$(\Gamma',f)$ is called {\it fractal}
(or recurrent)  if $f(\Gamma')=\Gamma$. A self-similar
action of $\Gamma$ on some
$A^\omega$  is called {\it fractal} if it is transitive and the corresponding self-similar datum is fractal, see \cite{N} section 2.8. Obviously a  fractal action has dense orbits.

The group $\Gamma$ is called {\it fractal} 
(respectively {\it freely fractal})
if $\Gamma$ admits  a faithfull (respectively free) fractal action on some $A^\omega$.

\section{Rational points of a torus}

We are going to prove Theorem 1, about the rational points of
algebraic tori.

For the whole chapter, let $\bf H$ be an algebraic
 torus defined over $\Q$ and let $X({\bf H})$ be the group of characters of $\bf H$.
For a number field $K$, let's denote by
$\Gal(K):=\Gal(\overline\Q/K)$ its
absolute Galois group. The group $X({\bf H})$ is a 
$\Gal(\Q)$-module which is 
isomorphic to $\Z^r$  as an abelian group,
where $r= \dim {\bf H}$.  The {\it splitting field} of $\bf H$ is the smallest Galois extension $L$ of $\Q$ such that $\Gal(L)$ acts trivially on $X({\bf H})$, or, equivalently  such that $\bf H$ is $L$-isomorphic to $\G_m^r$, where $\G_m$ denotes the multiplicative group. Moreover, we have 

\centerline{$\chi(h)\in L^*$}

\noindent for any $\chi\in X({\bf H})$ and any $h\in {\bf H}(\Q)$.

 Let ${\cal O}$ be 
the ring of integers of $L$. Recall that a {\it fractional ideal} is a nonzero finitely generated ${\cal O}$-submodule of $K$. A fractional ideal $I$ is called {\it integral} if $I\subset {\cal O}$. If the fractional ideal $I$ is integral and 
$I\neq{\cal O}$, then $I$  is merely an  ideal of ${\cal O}$.

Let  ${\cal I}$ be the set of all
fractional ideals and ${\cal I}^+$ be the subset of
all integral ideals. Given  $I$ and $J$ in ${\cal I}$, 
their {\it product} is the ${\cal O}$-module generated by all products $ab$ where $a\in I$ and $b\in J$. Since 
${\cal O}$ is a Dedekind ring,
we have

\centerline{${\cal I}\simeq\oplus_{\pi\in{\cal P}}\, \Z \,[\pi]$} 

\centerline{${\cal I}^+\simeq\oplus_{\pi\in{\cal P}} \,\Z_{\geq 0}\, [\pi]$,}

\noindent where ${\cal P}$ is the set of prime ideals of 
${\cal O}$.
Indeed the additive notation is used for
for the group ${\cal I}$ and the monoid ${\cal I}^+$: 
view as an element of
${\cal I}$ the fractional ideal
$\pi_1^{m_1}\dots \pi_n^{m_n}$ will be denoted 
as $m_1 [\pi_1]+\dots +m_n [\pi_n]$. 

Since $\Gal(L/\Q)$ acts by permutation on ${\cal P}$, 
${\cal I}$ is a   $\Z \Gal(L/\Q)$-module. For 
$S\subset {\cal P}$, set 

\centerline{
${\cal I}_S=\oplus_{\pi\in{\cal P}\setminus S}\, \Z \,[\pi]$.}

\begin{lemma}\label{Cebotarev} Let $S\subset {\cal P}$ be a finite 
 subset and let $r>0$ be an integer.

The $\Gal(L/\Q)$-module ${\cal I}$ contains a free
$\Z \Gal(L/\Q)$-module $M(r)$ of rank $r$ such that 

(i) $M(r)\cap {\cal I}^+=\{0\}$, and

(ii) $M(r)\subset {\cal I}_S$.
\end{lemma} 

\begin{proof} Let $S'$ be the set of all prime numbers which are divisible by some $\pi\in S$.
Let $\Sigma$ be the set of prime numbers
$p$ that are completely split in $K$, i.e. such that
${\cal O}/p{\cal O}\simeq \F_p^{[L:\Q]}$. 
For $p\in\Sigma$,  let $\pi\in {\cal P}$ be a prime
ideal over $p$.  When $\sigma$ runs over $\Gal(L/\Q)$ 
the ideals $\pi^\sigma$ are all distinct, and therefore $[\pi]$ generates a free $\Z\Gal(L/\Q)$-submodule of rank one
in ${\cal I}$.

By Cebotarev theorem, the set $\Sigma$ is infinite. Choose $r+1$ distinct prime numbers
$p_0,\dots p_{r}$ in $\Sigma\setminus S'$, and let $\pi_0,\dots,\pi_{r}\in{\cal P}$ such that
${\cal O}/\pi_i=\F_{p_i}$. For $1\leq i\leq r$, set 
$\tau_i=[\pi_i]-[\pi_0]$ and let
$M(r)$ be the $\Z \Gal(L/\Q)$-module generated by 
$\tau_1,\dots,\tau_r$. 

Obviously, the $\Gal(L/\Q)$-module $M(r)$ 
is free of rank $r$ and $M(r)\subset {\cal I}_S$. 
It remains to prove that $M(r)\cap {\cal I}^+=\{0\}$.
Let 

\centerline{$A=\sum\limits_{1\leq i\leq r, \sigma\in \Gal(L/\Q)}\, m_{i,\sigma}\,
\tau_i^\sigma$}

\noindent be an element of $M(r)\cap {\cal I}^+$. 
We have $A=B-C$, where

\centerline{$B=\sum\limits_{1\leq i\leq r, \sigma\in \Gal(L/\Q)}\, m_{i,\sigma}\,
[\pi_i^\sigma]$, and}

\centerline{$C=\sum\limits_{\sigma\in \Gal(L/\Q)}\,
(\sum\limits_{1\leq i\leq r}\, m_i^\sigma)
[\pi_0^\sigma]$.}

\noindent Thus the condition $A\in {\cal I}^+$ implies
that 

\centerline{$m_i^\sigma\geq 0$, for any $1\leq i\leq k$ and 
$\sigma\in \Gal(L/\Q)$,  and}

 \centerline{$\sum\limits_{1\leq i\leq r}\, m_i^\sigma\leq 0$,
 for any $\sigma\in \Gal(L/\Q)$.}

\noindent Thus  all the integers $m_i^\sigma$ vanish.
Therefore $M(r)$  intersects ${\cal I}^+$ trivially.

\end{proof}

For $\pi\in{\cal P}$, let $v_\pi:L^*\to\Z$ be  the corresponding valuation. 

\begin{lemma}\label{free}  Let $S\subset {\cal P}$ be a finite 
$\Gal(L/\Q)$-invariant subset
and let $r>0$ be an integer.

The $\Gal(L/\Q)$-module $L^*$ contains a free
$\Z \Gal(L/\Q)$-module $N(r)$ of rank $r$ such that 

(i) $N(r)\cap {\cal O}=\{1\}$, and

(ii) $v_\pi(x)=0$ for any $x\in N(r)$ and any $\pi\in S$.
\end{lemma}

\begin{proof} 
Set
$L^*_S=\{x\in L^*\vert v_\pi(x)=0,  \,\forall \pi\in S\}$
and let $\theta: L^*_S\rightarrow {\cal I}_S$ be the map 
$x\mapsto\sum_{\pi\in {\cal P}}\,v_\pi(x)\,[\pi]$. 

By Lemma \ref{Cebotarev},  ${\cal I}_S$ contains a free
$\Z \Gal(L/\Q)$-module $M(r)$ of rank $r$ such that 
$M(r)\cap {\cal I}^+=\{0\}$. Let's remark that $\Coker\,\theta$ 
is a subgroup of  the class group
$\Cl(L)$ of $L$. 
Since, by Dirichelet Theorem, $\Cl(L)$ is finite,
there is a positive integer $d$ such that $d.M(r)$ lies in the image of
$\theta$. Since $M(r)$ is free, there is a free 
$\Z \Gal(K/\Q)$-module $N(r)\subset L^*_S$ 
of rank $r$ which is a lift of $dM(r)$, i.e. such that $\theta$  
induces an isomorphism
$N(r)\simeq d.M(r)$. Since $\theta({\cal O}\setminus 0)$ lies in
${\cal I}^+$, we have $\theta(N(r)\cap {\cal O})=\{0\}$.
 It follows that $N(r)\cap {\cal O}=\{1\}$.

The second assertion follows from the fact that $N(r)$ lies in 
$L^*_S$.
\end{proof}

 For $\pi\in {\cal P}$, let ${\cal O}_\pi$ and $L_\pi$ be the $\pi$-adic completions  of ${\cal O}$ and $L$. 
Let $x,\,y \in L$ and let $n>0$ be an integer. In what follows,
the congruence 

\centerline{$x\equiv y$ modulo $n{\cal O}_\pi$}

\noindent  means $x_\pi\equiv y_\pi\mod n {\cal O}_\pi$, where 
$x_\pi$ and $y_\pi$ are the images of $x$ and $y$ in $L_\pi$. 

The case  $n=1$ of the next statement will be used in
further sections. In such a case,
Assertion (ii) is tautological.

\begin{thm} \label{Main0} Let ${\bf H}$ be an algebraic torus defined over $\Q$,
and let $L$ be its splitting field.
Let $n>0$ be an integer and let $S\subset {\cal P}$ be 
the set of prime divisors of $n$.

There exists $h\in {\bf H}(\Q)$ such that

(i) $\chi(h)$ is not an algebraic integer, 
for any non-trivial  $\chi\in X({\bf H})$, and

(ii) $\chi(h) \equiv 1 \mod n{\cal O}_\pi$
for any $\chi\in X({\bf H})$ and any $\pi\in S$.
\end{thm}

\begin{proof} {\it Step 1.} First an element $h'\in {\bf H}(\Q)$
satisfying  Assertion (i) and

\centerline {(iii) $v_\pi(\chi(h'))=0$, for any $\pi\in S$ and any 
$\chi\in X({\bf H})$}

\noindent is found.

The abelian group $X({\bf H})$ is free of rank $r$
where $r=\dim\,{\bf H}$. Therefore, the comultiplication
$\Delta:X({\bf H})\rightarrow X({\bf H})\otimes \Z\Gal(L/\Q)$ provides an embedding of $X({\bf H})$ into
a free $\Z\Gal(L/\Q)$-module of rank $r$. By lemma \ref{free}, there
a free $\Z \Gal(L/\Q)$-module $N(r)\subset L^*_S$ of rank $r$
with $N(r)\cap {\cal O}=\{1\}$. 
Let 

\centerline{$\mu:X({\bf H})\otimes \Z\Gal(L/\Q)\to N(r)$,} 

\noindent be an isomorphism of $\Z \Gal(L/\Q)$-modules,
and set $h'=\mu\circ\Delta$.

 Since
${\bf H}(\Q)=\Hom_{\Gal(L/\Q)}(X({\bf H}),L^*)$, the embedding $h'$ is indeed an element of ${\bf H}(\Q)$. Viewed as a map
from $X({\bf H})$ to $L^*$, $h'$ is the morphism
$\chi\in X({\bf H})\mapsto \chi(h')$.

Since $\Image\,h'\cap{\cal O}=1$ and $h'$ is injective,
$\chi(h')$ is not an algebraic integer if $\chi$ is a non-trivial character. Since $\Image\,h'\subset L^*_S$, we have
 $v_\pi(\chi(h'))=0$ for any  $\chi\in X({\bf H})$. Therefore
 $h'$ satisfies Assertions (i) and (iii).
 
 \noindent{\it Step 2.} Let $\chi_1,\dots,\chi_r$ be a basis of
 $X({\bf H})$. Since $v_\pi(\chi_i(h'))=0$ for any $\pi\in S$,
 the element $\chi_i(h') \mod n{\cal O}_\pi$ is an inversible element in the finite ring ${\cal O}_\pi/n{\cal O}_\pi$. Therefore
 there are positive  integers $m_{i,\pi}$ such that
 
 \centerline
 {$\chi_i(h')^{m_{i,\pi}}\equiv 1 \mod n{\cal O}_\pi$,}
 
\noindent for all $1\leq i\leq r$ and all $\pi\in S$.
Set  $m=\lcm (m_{i,\pi})$
and set $h=h'^m$. Obviously $h$ satisfies Assertion (i).
 Moreover we have
$\chi_i(h)\equiv 1 \mod n{\cal O}_\pi$, for all $\pi\in S$ and all
$1\leq i\leq r$, and therefore $h$ satisfies Assertion (ii) as well.
 
\end{proof}

\section{Special Gradings}

Let $\fn$ be a finite dimensional 
Lie algebra defined over $\Q$ and let $\fz$ be its center. 
The relations between the gradings of $\C\otimes \fn$ and
the  automorphisms of $\fn$  are investigated now.

{\it The following important definitions
 will be used  in the whole paper}.
Let ${\cal S}(\fn)$ 
(respectively ${\cal V}(\fn)$) be the set
of all $f\in\Aut\,\fn$ such that
$\Spec\,f\vert_{\fz}$  (respectively $\Spec\,f$)
contains no algebraic integers. Moreover
let ${\cal F}(\fn)$ be the set of all
$f\in {\cal S}(\fn)$ such that all eigenvalues
of $f^{-1}$ are algebraic integers. Also set
${\cal F}^+(\fn)={\cal F}(\fn)\cap {\cal V}(n)$.
Here, by eigenvalues
of a $\Q$-linear endomorphism $F$, we always mean the
eigenvalues of $F$ in $\overline\Q$.

For any field $K$ of characteristic zero, set 
$\fn^K=K\otimes \fn$ and $\fz^K=K\otimes \fz$.
Let ${\bf G}={\bf Aut}\,\fn$ be the algebraic group  of  automorphisms of $\fn$. By definition, ${\bf G}$ is  
defined over $\Q$, and we
have ${\bf G}(K)=\Aut\, \fn^K$ for any field $K$ of 
characteristic zero. The notation $\fn$ underlines that $\fn$  can be viewed as  the functor in Lie algebras $K\mapsto \fn^K$. 
Let ${\bf H}\subset{\bf G}$ be a 
maximal torus  defined over 
$\Q$, whose existence is proved in
\cite{Che}, see also  \cite{Bor},  Theorem 18.2.

By definition, a {\it $K$-grading} of $\fn$ is 
is a decomposition of $\fn^K$

\centerline{$\fn^K=\oplus_{n\in\Z}\,\fn^K_n$}

\noindent such that $[\fn^K_n, \fn^K_m]\subset \fn^K_{n+m}$ for all $n,\,m\in\Z$. 
A grading is called {\it special} 
(respectively {\it very special}) if
$\fz^K\cap \fn^K_0=0$
(respectively if $\fn^K_0=0$).
A grading is called {\it non-negative} 
(respectively {\it positive}) if
$\fn^K_n=0$ for $n<0$
(respectively $\fn^K_n=0$ for $n\leq 0$).

For any field 
$K$ of characteristic zero, a  $K$-grading of $\fn$ can be identified 
with  an algebraic group morphism 

\centerline{$\rho:\G_m\rightarrow {\bf G}$}

\noindent defined over $K$, where $\G_m$ denotes
the multiplicative group.

 Consider  the following two hypotheses

(${\cal H}_K$) 
The Lie algebra $\fn$ admits a special $K$-grading, 

(${\cal H}^0_K$) 
The Lie algebra $\fn$ admits a very special $K$-grading.

\begin{lemma}\label{hypo} Let $K$ be the splitting field of 
${\bf H}$. Up to conjugacy, any grading of
$\fn^\C$ is defined over $K$. In particular

(i) The hypotheses  ${\cal H}_{\C}$ and
${\cal H}_{\overline\Q}$ are equivalent.

(ii) The hypotheses  ${\cal H}^0_{\C}$ and
${\cal H}^0_{\overline\Q}$ are equivalent.
\end{lemma}

\begin{proof}   Let 

\centerline{$\fn^{\C}=\oplus_{n\in\Z}\,\fn_n^{\C}$}

\noindent be a  grading of $\fn^{\C}$ and 
let $\rho:\G_m\rightarrow {\bf G}$ be the corresponding 
algebraic group morphism. Since any maximal torus of 
${\bf G}$ is ${\bf G}(\C)$-conjugate  to $\bf H$,  it  can be assumed that  $\rho(\G_m)\subset {\bf H}$.

Let $X({\bf H})$ be the character group of ${\bf H}$. The group
morphism $\rho$ is determined by the dual morphism
$L: X({\bf H})\rightarrow\Z=X(\G_m)$. However,
$\Gal(K)$ acts trivially on $X({\bf H})$.
Thus $\rho$ is automaticaly defined over
$K$.

\end{proof}

\begin{lemma}\label {technical}
Let $\Lambda$ be a finitely generated abelian group and let
$S\subset\Lambda$ be a finite subset containing no element  of finite order. Then there exists a morphism 
$L:\Lambda\rightarrow\Z$ such that 

\centerline{$L(\lambda)\neq 0$ for any $\lambda\in S$.}
\end{lemma}

\begin{proof}  Let $F$ be the subgroup of finite order elements in
$\Lambda$. Using $\Lambda/F$ instead of $\Lambda$, it can be assumed that 
$\Lambda=\Z^d$ for some $d$ an $0\notin S$.
Let's choose a positive  integer $N$ such that 
$S \subset\, ]-N,N[^d$ and
let 
$L:\Lambda\rightarrow \Z$ be the function defined by

\centerline{$L(a_1,\dots,a_d)=\sum_{1\leq i\leq d} a_i N^{i-1}$.}

\noindent For any 
$\lambda=(a_1,\dots,a_d)\in S$, there is a smallest index $i$ with $a_i\neq 0$. We have $L(\lambda)=a_i N^{i-1}$ modulo $N^i$.
Since 
$\vert a_i \vert<N$,  it follows that $L(\lambda)\neq 0 \mod N^i$ and therefore $L(\lambda)\neq 0$.
\end{proof}

\begin{lemma}\label{root} 
Let $f\in {\bf G}(\Q)$. There is a $f$-invariant $\Z$-grading of 
$\fn^{\overline \Q}$
such that all eigenvalues of $f$ on $\fn_0^{{\overline \Q}}$ 
are roots of unity. 

In particular, if $\Spec\,f$ contains no root of unity, then
$\fn^{\overline \Q}$ admits a very special grading.
\end{lemma}

\begin{proof}
Let $\Lambda\subset \overline\Q^* $ be the subgroup generated by the 
$\Spec\,f$. For any $\lambda\in\Lambda$ denote
by $E_{(\lambda)}\subset\fn^{\overline\Q}$ the corresponding generalized eigenspace of $f$.
Let  $R$ be the set of all roots of unity in
$\Spec\,f$ and set $S=\Spec\,f\setminus R$.

By  Lemma \ref{technical},  
there is a morphism $L:\Lambda\rightarrow \Z$
such that $L(\lambda)\neq 0$ for any $\lambda\in S$.
Let ${\cal G}$ be the decomposition

\centerline{$\fn^{\overline \Q}=\oplus_{k\in\Z}\,\fn^{\overline \Q}_k$}

\noindent of $\fn^{\overline \Q}$ defined by
$\fn^{\overline \Q}_k=\oplus_{L(\lambda)=k}\, E_{(\lambda)}$.
Since $[E_{(\lambda)},E_{(\mu)}] \subset E_{(\lambda\mu)}$
and $L(\lambda\mu)=L(\lambda)+L(\mu)$ for any 
$\lambda,\,\mu\in\Lambda$, 
it follows that ${\cal G}$ is a grading of the Lie algebra 
$\fn^{\overline \Q}$. Moreover
we have 

\centerline{$\fn^{\overline \Q}_0=\oplus_{\lambda\in R}\, E_{(\lambda)}$,}

\noindent from which the lemma follows.
\end{proof}

\begin{lemma}\label{H} With the previous notations

(i) the Lie algebra $\fn^\C$ admits a special grading iff
  ${\cal S}(\fn)\neq\emptyset$.
  
(ii) the Lie algebra $\fn^\C$ admits a very special grading iff
  ${\cal V}(\fn)\neq\emptyset$.
\end{lemma}

\begin{proof} In order to prove Assertion (i), let's
consider the following assertion

(${\cal A}$) \hskip3cm {$H^0({\bf H}(\overline\Q), \fz^{\overline \Q})=0$.}

\noindent  The proof  is based on the following "cycle" of implications

\noindent\centerline{
$\fn^\C$ has a special grading $\Rightarrow ({\cal A})$
$\Rightarrow {\cal S}(\fn)\neq\emptyset
\Rightarrow \fn^\C$ has a special grading.}

\noindent
{\it Step 1: the existence of a special grading of
$\fn^\C$ implies  (${\cal A}$).}
By hypothesis and Lemma \ref{hypo}, 
$\fn^{\overline\Q}$ admits a special grading.  Let  
$\rho:\G_m \rightarrow {\bf G}$
be the corresponding  group morphism.
Since  all maximal tori of ${\bf G}$ are conjugate to ${\bf H}$, we can assume that $\rho(\G_m)\subset {\bf H}$. Therefore we have 

\centerline{$H^0({\bf H}({\overline \Q}),\fz^{{\overline \Q}})
\subset H^0(\rho({\overline\Q}^*),\fz^{{\overline \Q}})=0$.}

\noindent 
Thus Assertion ${\cal A}$ is proved.
 
\noindent {\it Step 2:  proof that (${\cal A}$) implies  
${\cal S}(\fn)\neq\emptyset$.} 
By Theorem 1, there exists $f\in {\bf H}(\Q)$ such that 
$\chi(f)$ is not an algebraic integer for any non-trivial character 
$\chi\in X({\bf H})$. If we assume (${\cal A}$), then $\Spec\,f\vert_{\fz}$ contains no algebraic integers and therefore ${\cal S}(\fn)\neq\emptyset$.

\noindent {\it Step 3: proof that ${\cal S}(\fn)\neq\emptyset$ implies  
the existence of a special grading.}
For any $f\in {\cal S}(\fn)$,
Since $\Spec\,f\vert_{\fz}$ contains no roots of unity.
It follows from Lemma \ref{root} that the Lie algebra
$\fn^{\overline \Q}$ (and therefore $\fn^\C$)  admits a
special grading.   Therefore ${\cal S}(\fn)\neq\emptyset$ implies  
the existence of a special grading.

The proof of Assertion (ii) is almost identical. Instead of 
$({\cal A})$, the "cycle" of implications uses the following assertion

(${\cal A}^0$) \hskip3cm {$H^0({\bf H}(\overline\Q), \fn^{\overline \Q})=0$.}
\end{proof}

\begin{lemma}\label{H+} 

The  following  are equivalent:

(i) the Lie algebra $\fn^\Q$ admits a non-negative special grading, 

(ii) the Lie algebra $\fn^\C$ admits a non-negative special grading, and

(iii) The set ${\cal F}(\fn)$ is not empty.
\end{lemma}
 
\begin{proof} 
{\it Proof that $(ii) \Rightarrow (iii)$.}
Let $\fn^\C=\oplus_{k\geq 0}\,\fn^\C_k$ be a non-negative special grading of $\fn^\C$ and let
$\rho:\G_m\to {\bf G}$ be the corresponding group morphism.
Up to conjugacy, we can assume that
$\rho(\G_m)\subset {\bf H}$. It follows that
the grading is defined over the splitting field $K$ of
${\bf H}$. 

Let $g_1\in {\bf H}(K)$ be the isomorphism
 defined by $g_1 x=2^k x$ if 
$x\in \fn_k^\C$. Set $n=[K:\Q]$ and let 
$g_1, g_2\dots g_n$ be the $\Gal(L/\Q)$-conjugates of $g_1$.
Since all $g_i$ belongs to ${\bf H}(K)$, the 
automorphisms $g_i$ commute. Hence the
product $g:=g_1\dots g_n$ is well defined
and $g$ belongs to ${\bf H}(\Q)$.
By hypotheses, all eigenvalues of $g_i$ are power of $2$,
and all eigenvalues of $g_i\vert_{\fz^{\C}}$ are 
distinct from $1$. Therefore all eigenvalues
of $g$ are integers, and all eigenvalues of
$g\vert_{\fz^{\C}}$ are $\neq \pm1$. It follows that
$g^{-1}$ belongs to ${\cal F}(\fn)$. Therefore
${\cal F}(\fn)\neq\emptyset$ 

\noindent {\it Proof that 
$(iii) \Rightarrow (i)$.}
Let $f\in{\cal F}(\fn)$ and set $g=f^{-1}$.  
Set $K=\Q(\Spec\,g)$ and let
$L:K^*\rightarrow \Z$ be the map defined by

\centerline {$L(x)=\sum_{p}\,v_p(N_{K/\Q}(x))$}

\noindent where the sum runs over all prime numbers $p$ and 
where $N_{K/\Q}:K^*\rightarrow \Q^*$ denotes the norm
map.

For any integer $k$, set

\centerline {$\fn_k^{\overline \Q}=\bigoplus\limits_{L(x)=k}E_{(x)}$}

\noindent where $E_{(x)}\subset \fn^{\overline\Q}$ denotes the  generalized  eigenspace associated to $x\in \Spec\,g$.
We have $[E_{(x)},E_{(y)}]\subset E_{(xy)}$ and 
$L(xy)=L(x)+L(y)$, for any $x,\,y\in K$. Therefore
the decomposition 

\centerline{$\fn^K=\oplus_{k\in\Z}\,\fn_k^{\overline \Q}$}

\noindent is a grading $\cal G$ of the Lie algebra 
$\fn^{\overline \Q}$.
Since the function $L$ is $\Gal(\Q)$-invariant, the
grading $\cal G$  is indeed defined over 
$\Q$. It remains to prove that $\cal G$ is non-negative and special.

Since any  $x\in \Spec\,g$ is an algebraic integer,
we have $L(x)\geq 0$ and the grading is non-negative.
Since no $x\in\Spec\,g\vert_\fz$ is an algebraic unit, 
we have $N_{K/\Q}(x)\neq \pm 1$
and  $L(x)>0$. Thus the grading is special,
what proves that $(iii)\implies (i)$.

\end{proof} 
 
\begin{lemma}\label{H++} 

The  following  are equivalent:

(i) the Lie algebra $\fn^\Q$ admits a positive grading, 

(ii) the Lie algebra $\fn^\C$ admits a positive grading, and

(iii) The set ${\cal F}^+(\fn)$ is not empty.
\end{lemma}

Since the proof is almost identical to the previous proof,
it will be skipped. The equivalence 
$(i) \Leftrightarrow (ii)$ also appears in \cite{Cor}.

\section{Height and relative complexity} 

For the whole chapter, $V$ denotes a finite dimensional
vector space over $\Q$. In this section, we define
the notion of the {\it height}  of the isomorphisms
$h\in GL(V)$ and the
notion of a {\it minimal lattice}.

\bigskip
\noindent{\it 4.1 Height, complexity and minimality}

\noindent
Let $h\in GL(V)$. Recall that a {\it lattice}
of $V$ is a finitely generated subgroup $\Lambda$ which
contains a basis of $V$. 
Let ${\cal D}(h)$ be the set
of all couple of lattices $(\Lambda, E)$ of $V$ such that
$E\subset \Lambda$ and $h(E)\subset \Lambda$. By definition,
the {\it height} of $h$, is the integer 

\centerline{
$\he (h):=\Min_{(\Lambda, E)\in{\cal D}(h)}\, [\Lambda:E]$.}

\noindent Let ${\cal D}_{min}(h)$ 
be the set of all couples $(\Lambda,E)\in {\cal D}(h)$
such that $[\Lambda:E]=\he(h)$.

Similarly, for a lattice $\Lambda$ of $V$, the {\it $h$-complexity} of 
$\Lambda$ is the integer

\centerline{
$\cp_h (\Lambda):=\Min_{(\Lambda, E)\in{\cal D}(h)}\, [\Lambda:E]$.}

\noindent It is clear that $\cp_h (\Lambda)=[\Lambda:E]$,
where $E=\Lambda\cap h^{-1}(\Lambda)$. The lattice $\Lambda$ is
called {\it minimal relative to $h$} if $\cp_h (\Lambda)=\he (h)$.

For the proofs, a technical  notion of relative height is needed.
Let $\End_h(V)$ be the commutant of $h$ and let
let $A\subset C(h)\subset\End_h(V)$ be a subring. 
By definition, an {\it $A$-lattice} $\Lambda$ means a lattice
$\Lambda$ which is an $A$-module.
Let  ${\cal D}^A(h)$ be the set of all couple of $A$-lattices 
$(\Lambda, E)$ in ${\cal D}(h)$. 
The {\it $A$-height} of $h$ is the
integer

\centerline{
$\he_A(h):=\Min_{(\Lambda, E)\in{\cal D}^A(h)}\, [\Lambda:E]$.}

\noindent 
Obviously, we have $\he_A(h)\geq \he(h)=\he_\Z(h)$. 
Let ${\cal D}^A_{min}(h)$)
be the set of all couples $(\Lambda,E)\in {\cal D}^A(h)$) such that 
$[\Lambda:E]=\he_A(h)$.

\bigskip
\noindent{\it 4.2 Height and filtrations}

\noindent Let $V$ be a finite dimensional vector space over
$\Q$ and let $h\in GL(V)$. Let $A$ be a subring of $\End_h(V)$ and 
let $A[h]$ be the subring of $\End_h(V)$ generated by $A$ and $h$.

\begin{lemma}\label{filt}  Let $0=V_0\subset V_1\subset\dots\subset V_n=V$
be a fitration of  $V$, where each vector space $V_i$ is a
$A[h]$-submodule. For $i=1$ to $n$, set $h_i= h_{V_i/V_{i-1}}$. Then we have

\centerline{$\he_A(h)\geq \prod_{1\leq i\leq n}\,\he_A(h_i)$.}

Moreover if $V\simeq \oplus V_i/V_{i-1}$ as a $A[h]$-module,  we have

\centerline{$\he_A(h)= \prod_{1\leq i\leq n}\,\he_A(h_i)$.}
\end{lemma}

\begin{proof} Clearly it is enough to prove the lemma for
$n=2$. Let $(\Lambda,E)\in {\cal D}^A_{min}(h)$. 
Set $\Lambda_1=\Lambda\cap V_1$,
$E_1=E\cap V_1$, $\Lambda_2=\Lambda/\Lambda _1$ and
$E_2=E/E _1$. We have

\centerline{$[\Lambda:E]=[\Lambda_1:E_1][\Lambda_2:E_2]$.}

Since $(\Lambda_1,E_1)\in {\cal D}^A(h_1)$  and 
$(\Lambda_2,E_2)\in {\cal D}^A(h_2)$, we have

\centerline{$\he_A(h)\geq \he_A(h_1)\,\he_A(h_2)$,}

\noindent what proves the first assertion.

Next, we assume that $V\simeq V_1\oplus V_2$ as 
a $A[h]$-module. Let 
$(\Lambda_1,E_1)\in {\cal D}_{min}^A(h_1)$,
$(\Lambda_2,E_2)\in {\cal D}_{min}^A(h_2)$ and
set $\Lambda=\Lambda_1\oplus\Lambda_2$ and
$E=E_1\oplus E_2$. We have

\centerline{$[\Lambda:E]=[\Lambda_1:E_1][\Lambda_2:E_2]=\he_A(h_1)\,\he_A(h_2)$.}

\noindent Therefore 
$\he_A(h)\leq \he_A(h_1)\,\he_A(h_2)$. Hence $\he_A(h)=\he_A(h_1)\,\he_A(h_2)$.

\end{proof}

Let $h\in GL(V)$ as before. Its {\it Chevalley decomposition}  $h=h_s h_u$ is uniquely defined by the following three conditions:
$h_s$ and $h_u$ commutes, $h_s$ is semi-simple and 
$h_u$ is unipotent.
 
 \begin{lemma}\label{Chevalley} We have
 
 \centerline{$\he(h)=\he(h_s)$.}
 \end{lemma}
 
 \begin{proof} By Lemma \ref{filt}, it can be assumed that the 
 $\Q[h]$-module $V$ is indecomposable. Therefore there is a vector
 space $V_0$, a semi-simple endomorphism $h_0\in\End(V_0)$ and
 an isomorphism
 
 \centerline{$V\simeq V_0\otimes\Q[t]/(t^n)$,}
 
 \noindent relative to which
 $h_s$ acts as $h_0\otimes 1$ and $h_u$ acts as $1\otimes t$.
 Let $(\Lambda_0,E_0)\in {\cal D}_{min}(h_0)$ and set
 $\Lambda=\Lambda_0\otimes\Z[t]/(t^n)$ and
 $E= E_0\otimes\Z[t]/(t^n)$. By Lemma \ref{filt}, we have
 $\he(h)\geq \he(h_s)=\he(h_0)^n$. Since $(\Lambda,E)\in{\cal D}(h)$ and 
 
 \centerline{$[\Lambda:E]=[\Lambda_0:E_0]^n=\he(h_0)^n$,}
 
 \noindent it follows that $\he(h)= \he(h_s)$

\end{proof}

\bigskip\noindent
{\it 4.3 Complexity of ${\cal O}(h)$-lattices}

\noindent
For any algebraic number $\lambda$, let ${\cal O}(\lambda)$ be the ring of integers of the number field $\Q(\lambda)$. Set
$\pi_\lambda=\{x\in {\cal O}(\lambda)\vert\, x\lambda\in 
{\cal O}(\lambda)\}$. Then $\pi_\lambda$ is an integral ideal  and its norm is the integer

\centerline{$d(\lambda):=\N_{\Q(\lambda)/\Q}(\pi_\lambda)=
\Card {\cal O}(\lambda)/\pi_\lambda$.}

Let $h\in GL(V)$ be semi-simple. 
Let $P(t)$ be its minimal
polynomial, let $P=P_1\dots P_k$ be its factorization
into irreducible factors. For $1\leq i\leq k$, set  
$K_i=\Q[t]/(P_i(t))$ and let ${\cal O}_i$ be the ring of
integers of the number field $K_i$. Set 
${\cal O}(h)=\oplus_{1\leq i\leq k}\,{\cal O}_i$. 

For each $\lambda\in \Spec\,h$, let
$m_\lambda$ be its multiplicity. Note that the functions
$\lambda\mapsto m_\lambda$ and $\lambda\mapsto d(\lambda)$ 
are $\Gal(\Q)$-invariant, so they can be viewed as functions
defined over $\Spec\,h/\Gal(\Q)$.

\begin{lemma}\label{added} Let $\Lambda$ be an
${\cal O}(h)$-lattice of $V$. Then

\centerline{$\cp_h(\Lambda)=\prod\,d(\lambda)^{m_{\lambda}}$,}

\noindent where the product runs over $\Spec h/\Gal(\Q)$.
\end{lemma}

\begin{proof} With the previous notations, let
$e_i$ be the unit of ${\cal O}_i$ and set 
$\Lambda_i=e_i\Lambda$. Since 
$\Lambda=\oplus_{1\leq i\leq k}\,\Lambda_i$, it is enough to
prove the lemma for $k=1$, i.e. when the minimal polynomial
of $h$ is irreducible.

Let $\lambda$ be one eigenvalue of $h$. 
With these new hypotheses, we have 
$\Q[h]/(P(t))\simeq \Q(\lambda)$,
${\cal O}(h)\simeq {\cal O}(\lambda)$  and $V$ is a 
vector space of dimension $m_\lambda$ over $\Q(\lambda)$, 
relative to which $h$ is identified with the multiplication by 
$\lambda$. We have

\centerline{ $r_\lambda \Lambda=\Lambda\cap h^{-1}\Lambda$.}

\noindent Since $\Lambda/r_\lambda{\cal I}\simeq 
({\cal O}(\lambda)/r_{\lambda})^{m_\lambda}$, it follows
that $\cp_h(\Lambda)=d(\lambda)^{m_{\lambda}}$.

\end{proof}

\bigskip\noindent
{\it 4.4 Computation of the height}

\noindent Let
$h\in GL(V)$ be semi-simple.

\begin{lemma}\label{formula}  We have

\centerline{$\he (h)=\prod\,d(\lambda)^{m_{\lambda}}$,}

\noindent where the product runs over $\Spec h/\Gal(\Q)$.
\end{lemma}

\begin{proof}  Using Lemmas \ref{Chevalley} and Lemma \ref{filt}, we can be assumed $V$ is a simple $\Q[h]$-module, and let $n$ be its dimension. The eigenvalues
 $\lambda_1,\dots,\lambda_n$  of $h$ are conjugate by
 $\Gal(\Q)$. Under these simplifying hypotheses, the formula to be proved is
 
 \centerline{$\he (h)=d(\lambda_1)$.}

\noindent {\it Step 1: scalar extension.}
 Set $K=\Q(\lambda_1,\dots,\lambda_n)$, let $U=K\otimes V$,
 let $\tilde h = 1\otimes h$ be the extension of $h$ to $U$
 and let $\{v_1,\dots,v_n\}$ be a $K$ basis of $U$ such that
 $\tilde h.v_i=\lambda_i\,v_i$. We have $U=\oplus_{1\leq i\leq n} U_i$, where $U_i=K\,v_i$.

Let $\cal O$  be the ring of integers of $K$. For each 
$1\leq i\leq n$,
set $\tilde h_i=h\vert _{U_i}$. Since each $U_i$ is
a ${\cal O}[\tilde h]$-module, Lemma \ref{filt} shows that

 \centerline{$\he_{\cal O}(\tilde h)= \prod_{1\leq i\leq n}\,
 \he_{\cal O}(\tilde h_i)$.}
 
  Next, the integers $\he_{\cal O}(\tilde h_i)$ are computed.
 Let $\Lambda_i\subset U_i$ be any 
 ${\cal O}$-lattice. Since ${\cal O}$ contains
 ${\cal O}(\lambda_i)={\cal O}(\tilde h_i)$, 
 it follows from Lemma \ref{added} that
 
 \centerline{$\cp_{\tilde h_i}(\Lambda_i)=d(\lambda_i)^r$}
  
\noindent where $r=\rk_{\cal O}(\Lambda_i)=[K:\Q(\lambda_i)]$.  Hence we have
$\he_{\cal O}(\tilde h_i)=d(\lambda_i)^{[K:\Q(\lambda_i)]}$.
It follows that

\centerline{
$\he_{\cal O}(\tilde h)= \prod_{1\leq i\leq n}\,
d(\lambda_i)^{[K:\Q(\lambda_i)]}=d(\lambda_1)^{[K:\Q]}$}

 \noindent {\it Step 2: end of the proof.}
 Now let   $(\Lambda, E)\in {\cal D}_{min}(h)$. Set
 $\tilde \Lambda={\cal O}\otimes \Lambda$ and 
 $\tilde E={\cal O}\otimes E$. Since $\tilde E$ is an
 ${\cal O}$-module, we have
 $[\tilde \Lambda:\tilde E]\geq \he_{\cal O}(\tilde h)$.
 It follows that
 
\centerline{ $\he(h)^{[K:\Q]}=[\Lambda:E]^{[K:\Q]}
=[\tilde \Lambda:\tilde E]\geq \he_{\cal O}(\tilde h)=d(\lambda)^{[K:\Q]}$.}

\noindent Thus we have   $d(\lambda_1)\leq \he (h)$. By 
Lemma \ref{added}, we have 
$\he_{{\cal O}(h)} (h)=d(\lambda_1)$. It follows that
 
 \centerline{$d(\lambda_1)\leq \he (h)
 \leq \he_{\cal O(h)} (h)=d(\lambda_1)$,}
 
 \noindent what proves the formula.
  
 \end{proof}

\bigskip 
\noindent{\it Remark:}  In number theory, the Weil height 
 of an algebraic number $\lambda$ is 
 $H(\lambda)=\theta d(\lambda)^{1/n}$, where $\theta$ involves
 the norms at infinite places.  Therefore $\he(h)$ is essentially
 the Weil's height of $h$, up to the factor at infinite places. 
 \bigskip

\bigskip\noindent
{\it 4.5 A simple criterion of minimality} 

\noindent   
 An obvious consequence of
Lemmas \ref{added} and \ref{formula} is

\begin{lemma}\label{crit1} Let $h\in GL(V)$ be semi-simple
and let $\Lambda$ be an ${\cal O}(h)$-lattice of 
$V$.  Then  $\Lambda$ is  minimal relative to $h$.
\end{lemma}

\section{Malcev's Theorem and self-similar data}

In this chapter, we recall Malcev's Theorem. 
Then we collect some related results, which are due
to Malcev or  viewed as folklore results. Then it is easy
to characterize the self-similar data for FGTF nilpotent groups.

\bigskip
\noindent {\it 5.1 Three types of lattices}

\noindent Let $\fn$ be a  finite dimensional be a  nilpotent Lie algebra  over $\Q$. The Lie algebra $\fn$ is endowed with two 
group structures, the addition and the
the Campbell-Hausdorff product. To avoid confusion,
the Campbell-Hausdorff product is called the 
{\it multiplication} and it is denoted accordingly.

A {\it multiplicative subgroup} $\Gamma$ of $\fn$ means a subgroup relative to the Campbell-Hausdorff product. 
In general, a multiplicative subgroup
$\Gamma$ is not an additive subgroup of $\fn$. However,
notice that $\Z.x\subset\Gamma$ for any $x\in \Gamma$, because
$x^n=nx$ for any  $n\in\Z$.  

A finitely generated multiplicative subgroup $\Gamma$ is called a {\it multiplicative lattice} if
$\Gamma\mod [\fn,\fn]$ generates the $\Q$-vector space 
$\fn/[\fn,\fn]$,
or, equivalently, if $\Gamma$ generates the Lie algebra $\fn$.  
Let $N$ be the CSC nilpotent Lie group with Lie algebra 
$\fn^R=\R\otimes\fn$. A discrete subroup $\Gamma$ of  $N$ is called a
{\it cocompact lattice} if $N/\Gamma$ is compact. 

It should be noted that three distinct notions of lattices will be 
used in the sequel: the additive lattices, the multiplicative lattices and the cocompact lattices. When it is used alone, a lattice is
always an additive lattice. This very commoun  terminology 
could be confusing: the reader should read "multiplicative lattice" or "cocompact lattice" as single words.

\bigskip
\noindent {\it 5.2 Malcev's Theorem}

\noindent Any multiplicative
lattice $\Gamma$ of a finite dimensional nilpotent Lie algebra over $\Q$ is a  FGTF nilpotent group. Conversely, Malcev proved in \cite{Ma}

\begin{Malcevthm} Let $\Gamma$ be a FGTF nilpotent group.

1. There exists a unique nilpotent Lie algebra $\fn$ over $\Q$ wich contains $\Gamma$ as a multiplicative lattice.

2. There exists a unique CSC nilpotent Lie 
group $N$ which contains $\Gamma$ as a cocompact lattice.

3. The Lie algebra of $N$ is $\R\otimes \fn$.

\end{Malcevthm}

The Lie algebra $\fn$ of the previous theorem will be called
the {\it Malcev Lie algebra} of $\Gamma$.

\bigskip
\noindent {\it 5.3 The coset index} 

\noindent From now on, let $\fn$ will be a finite dimensional
nilpotent Lie algebra. 
The coset index, which is defined now,
generalizes the notions of indices for additive lattices and 
for  multiplicative lattices.

A subset $X$ of $\fn$ is called a {\it coset union}
if $X$ is a finite union of $\Lambda$-coset for some
additive lattice $\Lambda$.

Recall that the {\it nilpotency index} of $\fn$  is the smallest integer $n$ such that $C^{n+1}\fn=0$, where $(C^n\,\fn)_{n\geq 0}$ is its descending central series. The following lemma is easily
proved  by induction on the nilpotency index of $\fn$.
 
\begin{lemma}\label{coset} Any  multiplicative lattice 
$\Gamma$ of $\fn$ is a coset union.
\end{lemma}

Let $X\supset Y$ be two coset unions in $\fn$.
Obviously, there is a lattice $\Lambda$ such that $X$ and $Y$ are
both a finite union of $\Lambda$-coset. The {\it coset index} 
of $Y$ in $X$ is the number

\centerline{  $[X:Y]_{coset}=
{\Card\,X/\Lambda\over \Card\,Y/\Lambda}$}

\noindent The numerator and denominator of the previous expression
depends on the choice of $\Lambda$, but $[X:Y]_{coset}$ is well defined. In general, the coset index is {\it not} an integer.
Obviously if $\Lambda\supset\Lambda'$ are additive lattices in
$\fn$, we have 

\centerline {$[\Lambda:\Lambda']_{coset}=[\Lambda:\Lambda']$.}

Similarly, for multiplicative lattices there is

\begin{lemma}\label{index} 
Let  $\Gamma\supset\Gamma'$ be multiplicative  lattices in
$\fn$,  we have 

\centerline {$[\Gamma:\Gamma']_{coset}=[\Gamma:\Gamma']$.}

\end{lemma}

 The proof, done by induction on the nipotency index of $\fn$,
 is skipped.

\bigskip
\noindent {\it 5.4 Morphims of FGTF nilpotent groups}

\begin{lemma} \label{extension} Let $\Gamma$, $\Gamma'\subset\fn$ be multiplicative
lattices in $\fn$ and let $f:\Gamma'\to\Gamma$ be a group morphism.
Then $f$ extends uniquely to a Lie algebra morphism

\centerline{$\tilde f:\fn\to\fn$.}

\noindent Moreover $\tilde f$ is an isomorphism if $f$ is  injective.
\end{lemma}

When $f$ is an isomorphism, the result is due to Malcev,
see \cite{Ma}, Theorem 5. In general, the lemma is a folklore result and  it is implicitely used in Homotopy Theory, see e.g. \cite {BK}. Since  we did not found a precise reference,  a proof, essentially based on Hall's collecting formula 
(see Theorem 12.3.1 in \cite{H}), is now
provided. 

\begin{proof} 
 Let $x\in\fn$. Since $\Gamma$ contains an additive lattice by Lemma \ref{coset}, we have
$m\Z x\subset\Gamma$ for some $m>0$.
Thus there is a unique  map
$\tilde f:\fn\to\fn$  extending $f$ such that
$\tilde f(nx)=n\tilde f(x)$ for any $x\in\fn$ and $n\in\Z$.
It remains to prove that 

\centerline{$\tilde f(x+y)=\tilde f(x)+\tilde f(y)$, and
$\tilde f([x,y])=[\tilde f(x),\tilde f(y)]$,}

\noindent for any $x,\,y\in \fn$.

Let $n$ be the nilpotency index of $\fn$.
 Set ${\cal L}(2,n)={\cal L}(2)/C^{n+1}{\cal L}(2)$,
 where   ${\cal L}(2)$ denotes the free Lie algebra over $\Q$ freely generated by  $X$ and  $Y$. 
 Let $\Gamma(2,n)\subset {\cal L}(2,n)$ be the multiplicative subgroup  generated by $X$ and $Y$. 
 
As before,  
$m(X+Y)$ and $m [X,Y]$ belongs to $\Gamma(2,n)$
for some $m>0$. Thus there are  $w_1,\,w_2$ in the free group over two generators, such that

\centerline{$w_1(X,Y)=m(X+Y)$ and $w_2(X,Y)=m [X,Y]$.}

Since ${\cal L}(2,n)$ is a free in the category of
nilpotent Lie algebras of nilpotency index $\leq n$, we have

\centerline{$w_1(x,y)=m(x+y)$ and $w_2(x,y)=m [x,y]$}

\noindent for any $x,y\in \fn$. From this it follows easily that 
$\tilde f$ is a Lie algebra morphism.
\end{proof}

\bigskip
\noindent {\it 5.6 Self-similar data for FGTF nilpotent groups}

\noindent Let $\fz$ be the center of $\fn$. Recall
that ${\cal S}(\fn)$ (respectively ${\cal V}(\fn)$)
 is the set of all  $f\in\Aut\,\fn$ such that $\Spec\,f\vert_\fz$ 
(respectively $\Spec\,f$)
contains no algebraic integers.
Let $\Gamma\supset \Gamma'$ be multiplicative lattices of
$\fn$, let $f:\Gamma'\to\Gamma$ be a morphism and let
$\tilde f:\fn\to\fn$ be its extension.

\begin{lemma}\label{ssdatum} Let's assume that $f$ is injective. Then

(i) $(\Gamma',f)$ is a self-similar datum
iff $\tilde f$ belongs to ${\cal S}(\fn)$,

(ii) $(\Gamma',f)$ is a free self-similar datum
iff $\tilde f$ belongs to ${\cal V}(\fn)$

(iii) if $(\Gamma',f)$ is a fractal datum, then $f$
belongs to ${\cal F}(\fn)$.
\end{lemma}

\begin{proof}   
Let $V$ be a finite dimensional vector space over $\Q$
and let $f\in GL(V)$. We will repeatedly use the fact that
$\Spec\,f$ contains an algebraic integer iff
$V$ contains a finitely generated subgroup $E\neq 0$
such that $f(E)\subset E$.

\smallskip
\noindent {\it Proof of Assertion (i).}
 Since 
$\Gamma'$ contains a set of generators of
$\fn$, the subgroup $Z(\Gamma'):=\Gamma'\cap\fz$
is the center of $\Gamma'$. Let $K$ be the $f$-core
of the virtual endomorphism $(\Gamma',f)$.

Let's assume that $(\Gamma',f)$ is not a self-similar datum.
Since  $K\neq 1$, the additive group
$K\cap Z(\Gamma')$ is  non-trivial, finitely generated  and 
$\tilde f$-invariant.
Therefore ${\tilde f}\notin {\cal S}(\fn)$.

Conversely let's assume that ${\tilde f}\notin {\cal S}(\fn)$.
Then there is a nonzero
finitely generated subgroup $E\subset \fz$ such that
$\tilde f(E)\subset E$. By Lemma \ref{coset},
$Z(\Gamma')$ is an additive lattice of $\fz$. Therefore
we have   $mE \subset Z(\Gamma')$ for some $m>0$. Since $K$ contains $mE$, it follows that 
$(\Gamma',f)$ is not a self-similar datum.

\smallskip
\noindent {\it Proof of Assertion (ii).}
Let $A\subset\Gamma$ be a set of representatives 
of $\Gamma/\Gamma'$. Let's consider the action
of $\Gamma$ on $A^\omega$  associated
with the virtual endomorphism $(\Gamma',f)$.

Let's assume that ${\tilde f}\notin {\cal V}(\fn)$.
Then there is a nonzero
finitely generated abelian subgroup $F\subset \fn$ such that
$\tilde f(F)\subset F$. As before, it can be assumed
$F$ lies in $\Gamma'$. Let $e\in A$ be
the representative of the trivial coset and let 
$e^{\omega}=ee\dots$ be the infinite
word over the single letter $e$. Since
$f(F)\subset F$, it follows that
$\gamma(e^{\omega})=e^{\omega}$ for any 
$\gamma\in F$. Hence $\Gamma$ does not act freely on
$A^\omega$.

Conversely, let assume that $\Gamma$ does not act freely on
$A^\omega$. Let's define inductively the subsets
${\cal H}(n)\subset\Gamma$ by ${\cal H}(1)=\cup_{a\in A}\,
a\Gamma'a^{-1}$ and

\centerline{${\cal H}(n+1)=
\{\gamma\in\Gamma\vert\,\exists\,a\in A:
a\gamma a^{-1}\in \Gamma' \land f(a\gamma a^{-1})
\in {\cal H}(n)\}$,}

\noindent for $n\geq 1$. Indeed 
${\cal H}(n)$ is the set of all
$\gamma\in\Gamma$ which have at least one fixed point on $A^n$.
It follows easily that ${\cal H}:=\cap_{n\geq 1}\,{\cal H}(n)$ is the set of all 
$\gamma\in\Gamma$ which have at least one fixed point on $A^\omega$.
There is an integer $k$ such that

\centerline{${\cal H}\subset C^k\fn$ but ${\cal H}\not\subset C^{k+1}\fn$.}

\noindent Let $\overline{\cal H}$ be the image of
${\cal H}$ in $C^k\fn/C^{k+1}\fn$ and let 
$F$ be the additive subgroup of $C^k\fn/C^{k+1}\fn$ generated
by $\overline{\cal H}$. Since $\Gamma$ lies in a lattice,
$F$ is  finitely generated. Moreover we have
$a x a^{-1}\equiv x\mod\,C^{k+1}\fn$, for any $x\in C^{k}\fn$
and $a\in A$. It follows that 
$\tilde f_k(\overline{\cal H})\subset \overline{\cal H}$, where
$\tilde f_k$ is the linear map  induced by $\tilde f$ on
$C^k\fn/C^{k+1}\fn$. Hence $\tilde f_k(F)\subset F$ and
$\Spec\, \tilde f_k$ contains an algebraic integer. Therefore
${\tilde f}\notin {\cal V}(\fn)$.

\smallskip
\noindent {\it Proof of Assertion (iii).}
Let $(\Gamma',f)$ be a fractal datum. Let
$\Lambda$ be the additive lattice generated by $\Gamma$. Since

\centerline{$\tilde f^{-1}(\Lambda)\subset \Lambda$,}

\noindent  all  $x\in \Spec \tilde f^{-1}$ are algebraic integers.
Therefore $\tilde f$ belongs to ${\cal F}(\fn)$.
\end{proof}

\section{Relative complexity of multiplicative lattices}

\noindent This chapter
is the mutiplicative analogue of ch. 4. The main result is
the refined criterion of minimality. Together with Theorem 1, it
is the main ingredient of the proof of 
Theorem \ref{Main1} and \ref{Main2}. 

 Throughout the whole chapter, $\fn$ is finite dimensional nilpotent Lie algebra over $\Q$, and $\fz$ is its center. 

\bigskip
\noindent {\it 6.1 Complexity of  multiplicative lattices}

\noindent Let $f\in \Aut\,\fn$ and let $\Gamma$ be a 
multiplicative lattice of
$\fn$. 
The {\it complexity} of $\Gamma$ relative to $f$
is the integer

\centerline{$\cp_f(\Gamma)=[\Gamma:\Gamma']$,}

\noindent where $\Gamma'=\Gamma\cap f^{-1}(\Gamma)$.
 The multiplicative lattice $\Gamma$ is
called {\it minimal} relative to $f$ if 
$\cp_f(\Gamma)=\he(f)$. Thanks to
Lemma \ref{index} the notation
$\cp_f(\Gamma)$  is unambiguous.

\begin{lemma}\label{multineq} Let $\Gamma$ be multiplicative lattices of 
$\fn$. Then we have

\centerline{$\cp_f(\Gamma)\geq \he(f)$.}
\end{lemma}

\begin{proof} The proof goes by induction on the nilpotency index of $\fn$.

Let $Z$ be the center of $\Gamma$. Set
$\Gamma'=\Gamma\cap f^{-1}(\Gamma)$, $Z'=Z\cap f^{-1}(Z)$,
$\overline{\Gamma}=\Gamma/Z$, $\overline{\Gamma'}=\Gamma'/Z'$.
Also set ${\overline \fn}=\fn/\fz$ and let 
$\overline f:{\overline \fn}\to{\overline \fn}$ 
and $f_0:\fz\to\fz$ be the isomorphisms induced by $f$.

By induction hypothesis, we have
$\cp_{\overline f}(\overline{\Gamma})\geq \he(\overline f)$
and therefore 

\centerline{$[\overline{\Gamma}:\overline{\Gamma'}]\geq \he(\overline f)$.}

\noindent By definition, we have 
$[Z:Z']=\cp_{f_0}\,Z\geq \he(f_0)$. Moreover by
Lemma \ref{formula} we have $\he(f)=\he(f_0) \he(\overline f)$.
It follows that

\centerline{$\cp_f\,\Gamma=[\Gamma:\Gamma']
=[Z:Z']\,[\overline{\Gamma}:\overline{\Gamma'}]
\geq \he(f_0)\he(\overline f)= \he(f)$,}

\noindent and the statement is proved.
\end{proof}

\bigskip  
\noindent{\it 6.2 A property of the minimal multiplicative lattices}  

\noindent
Let $\Gamma$ be a multiplicative lattice of $\fn$ and
let $h\in \Aut\,\fn$. For simplicity, let's assume that 
$h$ is semi-simple.

\begin{lemma}\label{mingood} The following assertions are equivalent

(i) $\Gamma$ is minimal relative to $h$, and

(ii) the virtual morphism $(\Gamma',h)$ is good,
where $\Gamma'=\Gamma\cap h^{-1}(\Gamma)$.

In particular, there is a multiplicative lattice
$\tilde\Gamma\subset\Gamma$ which is minimal relative to $h$.
\end{lemma}

\begin{proof} By Lemma \ref{coset}, $\Gamma$ is a coset union.
Any additive lattice  contains a
${\cal O}(h)$-module of finite index. Therefore there is
an ${\cal O}(h)$-lattice $\Lambda$ such
that $\Gamma$ is an union of $\Lambda$-cosets.

Let $\Gamma_0,\Gamma_1,\dots$ be the multiplicative lattices inductively defined by $\Gamma_0=\Gamma$,
$\Gamma_1=\Gamma'$  and
$\Gamma_{n+1}=\Gamma_n\cap h^{-1}(\Gamma_n)$ for $n\geq 1$. Similarly
let $\Lambda_0,\Lambda_1,\dots$ be the additive lattices  defined by $\Lambda_0=\Lambda$,
and $\Lambda_{n+1}=\Lambda_n\cap h^{-1}(\Lambda_n)$ for 
$n\geq 0$.

By  Lemma \ref{good}, the sequence $[\Gamma_n:\Gamma_{n+1}]$
is not increasing. By Lemma \ref{multineq}, we have
$[\Gamma_n:\Gamma_{n+1}]\geq \he(f)$.
Moreover, it follows from Lemma \ref{crit1} that
$[\Lambda_n:\Lambda_{n+1}]= \he(h)$ for all $n$.

Let's assume now that $\Gamma$ is minimal relative to
$h$. We have $[\Gamma_n:\Gamma_{n+1}]= \he(f)$ for all $n$, and therefore the virtual morphism $(\Gamma',h)$ is good.

Conversely, let's assume that  the virtual morphism $(\Gamma',h)$ is good. By hypotheses we have
$[\Gamma_0:\Gamma_n]= [\Gamma_0:\Gamma_1]^n$ and
$[\Lambda_0:\Lambda_n]= \he(h)^n$ for all $n\geq 1$.
It follows that

\centerline{
$[\Gamma_0:\Lambda_n]_{coset}=
[\Gamma_0:\Lambda_0]_{coset}\,\he(h)^n$.}

\noindent Since $\Gamma_n\supset\Lambda_n$, we have
$[\Gamma_0:\Gamma_n]\leq [\Gamma_0:\Lambda_n]_{coset}$.

and therefore

\centerline{
$[\Gamma_0:\Gamma_1]^n \leq 
[\Gamma_0:\Lambda_0]_{coset}\,\he(h)^n$,
for all $n\geq 0$.}

\noindent Hence
$[\Gamma_0:\Gamma_{1}]\leq\he(f)$. 
It follows from Lemma \ref{multineq}
that $[\Gamma_0:\Gamma_{1}]=\he(f)$, 
thus $\Gamma$
is minimal relative to $h$.

In order to prove the last assertion, notice that
the sequence $[\Gamma_n:\Gamma_{n+1}]$ is stationary for 
$n\geq N$, for some $N>0$.
Therefore $(\Gamma_{N+1},h)$ is a good virtual morphism
of $\Gamma_N$. Thus the subgroup 
$\tilde\Gamma=\Gamma_N$ is minimal relative to
$h$.

 \end{proof}

\bigskip\noindent
{\it 6.3 A refined criterion of minimality}

\noindent A refined version of  Lemma \ref{crit1} is now provided. 
Let $\Gamma$ be a multiplicative lattice in $\fn$ and
let $h\in\Aut\,\fn$ be semi-simple. Let $L$ be the field
generated by $\Spec h$, let ${\cal O}$ be its ring of integers and
let ${\cal P}$ be the set of prime ideals of ${\cal O}$. 

Let  $\Lambda$ be an  ${\cal O}(h)$-lattice and let $n>0$ be an integer. Let's assume that 

\centerline{$\Lambda\supset\Gamma$ and $\Gamma$ is an union of $n\Lambda$-cosets.}

\begin{lemma}\label{refined} Let $S$ be the set of divisors of $n$ in 
${\cal P}$. Assume  that

\centerline{$\lambda\equiv 1 \mod n{\cal O}_\pi$,}

\noindent for any $\lambda\in\Spec\,h$ and any $\pi\in S$. Then 
$\Gamma$ is minimal relative
to $h$.  

\end{lemma}

\begin{proof}  {\it Step 1.}  Since $\Spec\,h$ lies
in ${\cal O}_\pi$ for all $\pi\in S$, 
there exists a positive integer $d$, which is prime to $n$, such that
$d \lambda\in{\cal O}$ for all $\lambda\in{\cal O}$.
Moreover we can assume that $d\equiv 1\mod n$.

Let $\lambda\in \Spec\,h$.  We have $d \lambda\equiv 1 \mod n{\cal O}_\pi$ for all $\pi\in S$. Therefore we have

\centerline{$d\lambda\in 1 + n{\cal O}$,}

\noindent for all $\lambda\in  \Spec\,h$.
Set $H=dh$.  Since $\Spec\, dH$ and
$\Spec\, (H-1)/n$ lie in ${\cal O}$, 
it follows that

\centerline{$H\in {\cal O}(h)$ and $H\in 1 + n{\cal O}(h)$.}

\noindent {\it Step 2.}  Set $\Lambda'=\Lambda\cap h^{-1}\Lambda$.
Since all eigenvalues  of $h$ are units in ${\cal O}_\pi$ whenever 
$\pi$ divides $n$, the height of $h$ is prime to $n$. By 
Lemma \ref{crit1},
we have $[\Lambda:\Lambda']=\he(h)$. Therefore we get

\centerline{$\Lambda= \Lambda'+n\Lambda$.}

It follows that

\centerline{$\Gamma = \coprod\limits_{1\leq i\leq k}\, g_i+n\Lambda $}

\noindent for some $g_1,...,g_k\in  \Lambda'$,
where $k=[\Gamma:n\Lambda]$ and where $\coprod$ is the symbol of the disjoint union. Since $H(g_i)\equiv g_i\mod n\Lambda$, we get that
$h (g_i)\in g_i+n\Lambda\subset \Gamma$. Therefore we have

\centerline{$\Gamma'\supset \coprod\limits_{1\leq i\leq k}\, g_i+n\Lambda'$,}

Therefore we have $[\Gamma':n\Lambda']\geq k=[\Gamma:n\Lambda]$.
It follows that

\centerline {$[\Gamma:\Gamma']\leq [n\Lambda:n\Lambda']=\he(h)$.}

\noindent By Lemma \ref{multineq}, we have
$[\Gamma:\Gamma']=\he(h)$. Thus $\Gamma$ is minimal relative to $h$.
\end{proof}

\section{Proof of Theorems \ref{Main1} and \ref{Main2}}

\bigskip
\noindent {\it 7.1 Proof of Theorem 2 and 3.}

\noindent
Let $\fn$ be a finite dimensional
 nilpotent Lie algebra  over $\Q$
and let $\fz$ be its center and 
let $\Gamma$ be a multiplicative lattice of $\fn$.

\begin{thm}\label{Main1} The following assertions are equivalent

(i) The group $\Gamma$ is transitive self-similar,

(ii) the group $\Gamma$ is densely self-similar, and

(iii) the Lie algebra $\fn^\C$ admits a special grading.
\end{thm}

\begin{proof} 
Let's consider the following assertion

\noindent $({\cal A})$\hskip1cm ${\cal S}(\fn)\neq\emptyset$.

\noindent The implication $(ii)\Rightarrow (i)$ is tautological.
Together with the Lemmas \ref{hypo}(i) and \ref{H}(i), the following implications are already proved

\centerline{$(ii)\Rightarrow (i)\Rightarrow ({\cal A}) \Leftrightarrow (iii)$.}

\noindent Therefore, it is enough to prove that 
$({\cal A})\Rightarrow (ii)$.

\noindent{\it Step 1. Definition of some 
$h\in {\bf G}(\Q)$.}
Let's assume that ${\cal S}(\fn)\neq\emptyset$, and let 
$f\in {\cal S}(\fn)$. Since the semi-simple part of $f$ is
also in ${\bf G}(\Q)$, it can be assumed that $f$ is semi-simple.
Let ${\bf K}\subset {\bf G}$ be the Zariski closure of the subgroup
generated by  $f$  and set ${\bf H}={\bf K}^0$.

Let $\Lambda$ be the ${\cal O}(f)$-module generated
by $\Gamma$. By Lemma \ref{coset}, $\Gamma$ is a coset
union. Therefore $\Lambda$ is a lattice and 
$\Gamma$ is
an union of $n\Lambda$-coset for some positive integer $n$. 

 Let $X({\bf H})$ be the group of characters of
 ${\bf H}$, let $K$ be the splitting field of ${\bf H}$,
let ${\cal O}$ be the ring of integers of $K$,
let ${\cal P}$ be the set of prime ideals of 
${\cal O}$ and let $S$ be set set of all 
$\pi\in{\cal P}$ dividing $n$.

By Theorem 1, there exists $h\in {\bf H}(\Q)$ such that, 
for any non-trivial $\chi\in X$ we have

(i) $\chi(h)$ is not an algebraic integer, and

(ii) $\chi(h)\equiv 1\mod n{\cal O}_\pi$ for any $\pi\in S$.

\noindent
{\it Step 2.}
Let $\Gamma'=\Gamma\cap h^{-1}(\Gamma)$. We claim
that the virtual morphism $(\Gamma',h)$ is a good self-similar
datum.

Since  ${\bf K}\subset {\bf G}$ is the Zariski closure 
of the subgroup generated by  $f$, we have
$\Q[h]\subset\Q[f]$ and therefore 
$\Lambda$ is a ${\cal O}(h)$-lattice.
It follows from Lemma \ref{refined} that the virtual
endomorphism  $(\Gamma',h)$ is good.

Moreover, let $\Omega_0$ be the set of weights of ${\bf H}$ over 
$\fz^{\overline \Q}$. There is an integer $l$ such that
$f^l\in {\bf K}^0={\bf H}$. The spectrum of $f^l$ on
$\fz^{\overline \Q}$ are the numbers $\chi(f^l)$ when
$\chi$ runs over $\Omega_0$. Thus it follows that
$\Omega_0$ does not contain the trivial character,
hence $h$ belongs to ${\cal S}(\fn)$.

Therefore by Lemma \ref{ssdatum}, the virtual
endomorphism  $(\Gamma',h)$ is a good self-similar datum. Thus by
Lemma \ref{corr2}, $\Gamma$ is a densely self-similar group.

\end{proof}

\begin{thm}\label{Main2} The following assertions are equivalent

(i) The group $\Gamma$ is freely  self-similar,

(ii) the group $\Gamma$ is freely densely self-similar, and

(iii) the Lie algebra $\fn^\C$ admits a very special grading.
\end{thm}

\begin{proof} Let's assume Assertion (i). Let's
consider a free self-similar action of $\Gamma$
on some $A^\omega$ and let 
$A'$ be any $\Gamma$-orbit in $A$. Then the action
of $\Gamma$ on $A'^\omega$ is free transitive self-similar, thus   $\Gamma$ is freely transitive self-similar.

The rest of the proof is identical to the previous proof, except that

1)  the assertion $({\cal A})$ is replaced by
$({\cal A'})$: ${\cal V}(\fn)\neq\emptyset$,

2) the  Lemmas \ref{hypo}(ii) and \ref{H}(ii) are used instead of
Lemmas \ref{hypo}(i) and \ref{H}(i) in order to prove that
$(ii)\Rightarrow (i)\Rightarrow ({\cal A'}) 
\Leftrightarrow (iii)$,

3) the proof that ${\cal A'}\Rightarrow (ii)$ uses the weights 
of ${\bf H}$ and the eigenvalues of $f$ on $\fn$ instead of $\fz$.
\end{proof}

\bigskip
\noindent {\it 7.2 Manning's Theorem}

\noindent
Let $N$ be a CSC nilpotent
 Lie group $N$ and let $\Gamma$ be a cocompact lattice.
The manifold $M=N/\Gamma$ is called a {\it nilmanifold}.

A diffeomorphism
$f:M\to M$ is called an {\it Anosov diffeomorphism} if

(i) there is a continuous splitting of the tangent bundle $TM$ as
$TM=E_u\oplus E_s$ which is invariant by $df$, and

(ii) there is a Riemannian metric relative to which
$df\vert_{E_s}$ and $df^{-1}\vert_{E_u}$ are contracting.

For any $x\in M$,
$f$ induces a group automorphism  $f_*$ of
$\Gamma\simeq \pi_1(M)$.
By Lemma \ref{extension}, $f_*$ extends to an isomorphism
$\tilde f_*:\fn^{\R}\to\fn^{\R} $, where 
$\fn^{\R}$ is the Lie algebra of $N$. Strictly speaking,
$\tilde f_*$ is only defined up to an inner automorphism.
Since $f_*$ is well defined modulo the unipotent radical of 
$\Aut\,\fn^{\R}$, the set $\Spec f_*$ is unambiguously defined.

\begin{Manningthm} The set $\Spec f_*$ contains no root of unity.
\end{Manningthm}

See \cite{Ma}. Later on, 
A. Manning proved a much stronger result. Namely $\Spec f_*$ contains no eigenvalues of absolute value $1$, and $f$ is topologically conjugated to an Anosov automorphism, see \cite{Ma2}.

\bigskip
\noindent {\it 7.3 A Corollary for nilmanifolds with an Anosov diffeomorphim}

\begin{cor}\label{Anosov} 
Let $M$ be a nilmanifold endowed with an
Anosov diffeomorphism. Then 
$\pi_1(M)$ is freely densely self-similar.
\end{cor}

\begin{proof} By definition, we have  $M=N/\Gamma$, where
$N$ is a CSC nilpotent  Lie group
and $\Gamma\simeq \pi_1(M)$ is a cocompact lattice. Set 
$\fn^\R=\Lie\,N$ and $\fn^\C=\C\otimes\fn^\R$. By
Manning's Theorem and Lemma \ref{root}, $\fn^\C$ has a very special grading.
Therefore $\Gamma$ is freely densely self-similar by Theorem 3.
\end{proof}

\bigskip
\noindent
 {\it 7.4 Characterisation of fractal FGTF nilpotent groups} 

\noindent For completeness purpose, we will now 
investigate the non-negative gradings of $\fn^C$.
Unlike Theorems 2 and 3, the proof of Propositions 
\ref {nonneg} and \ref{positive}
are quite obvious.

Let $\fn^{\Q}$ be a finite dimensional nilpotent Lie algebra
and let $\Gamma$ be a multiplicative lattice in $\fn$.
Set $\fn^{\C}=\C\otimes \fn^{\Q}$.

\begin{prop} \label{nonneg} The following assertions are equivalent

(i) The group $\Gamma$ is fractal

(ii) $\fn^{\C}$ admits a non-negative special grading.

(iii) $\fn^{\Q}$ admits a non-negative special grading.
\end{prop}

\begin{proof} It follows from Lemma \ref{H+} that Assertions (ii) and (iii) are equivalent.

\noindent {\it Proof that  (i) $\Rightarrow$  (ii).}
By assumption, there is a fractal datum
$(\Gamma',f)$. Let $g:\Gamma\rightarrow \Gamma'$ be the inverse of $f$ and let
$\tilde g\in\Aut\,\fn$ be its unique extension.

Let $\Lambda\subset \fn$ be the additive subgroup generated by $\Gamma$. By Lemmas \ref{coset}, $\Lambda$ is an additive lattice. Since
we have $\tilde g(\Lambda)\subset \Lambda$, it follows that
all eigenvalues of $\tilde g$ are algebraic integers. 

Moreover $(\Gamma',g^{-1})$ is a self-similar datum, thus
$\Spec\,\tilde g^{-1}\vert_\fz$ contains no root of unity.
Therefore, by Lemma \ref{H+}, Assertion (ii) holds.

\noindent  {\it Proof that  (iii) $\Rightarrow$ (i).}
Let's assume Assertion (iii) and let

\centerline{$\fn^\Q=\oplus_{k\geq 0}\,\fn_k^\Q$}

\noindent be a non-negative special grading of $\fn^\Q$.

By Lemma \ref{coset}, $\Gamma$ lies in a lattice $\Lambda$.
 Since it is  possible to enlarge 
$\Lambda$, we can assume that

\centerline{$\Lambda=\oplus_{k\geq 0}\,\Lambda_k$, }

\noindent where $\Lambda_k=\Lambda\cap \fn_k^\Q$. 
Since $\Gamma$ is a coset union, there 
is an integer $d\geq 1$ such that  $\Gamma$ is an union of 
$d\Lambda$-cosets.

Let $g$ be the automorphism of $\fn^\Q$ defined by
$g(x)=(d+1)^k\,x$ if $x\in \fn_k^\Q$. 
We claim that $g(\Gamma)\subset \Gamma$. Let 
$x\in \Gamma$ and let $x=\sum_{k\geq 0}\,x_k$ be its decomposition into homogenous components. We have

\centerline{$g(x)=x+\sum_{k\geq 1}\,((d+1)^k-1)x_k$.}

\noindent By hypothesis each homogenous component $x_k$ belongs to 
$\Lambda$. Since $(d+1)^k-1$ is divisible by $d$,
we have $g(x)\in x+ d\Lambda\subset \Gamma$ and the claim is proved.

Set $\Gamma'=g(\Gamma)$ and let $f:\Gamma'\rightarrow \Gamma$ be the inverse of  $g$. It is clear that 
$(\Gamma',f)$ is a fractal datum for $\Gamma$, what proves Assertion (i).
\end{proof}

\begin{prop} \label{positive}
The following assertions are equivalent

(i) The group $\Gamma$ is freely fractal

(ii) $\fn^{\C}$ admits a positive grading.

(iii) $\fn^{\Q}$ admits a positive grading.
\end{prop}
 
 Since the proof is strictly identical, it will be skipped.

\section{Not self-similar  FGTF nilpotent groups
 and affine nilmanifolds}
 
 This section provides an example of a FGTF nilpotent group 
 which is not even self-similar, see subsection 8.6. The end of the section is about the Milnor-Scheuneman conjecture.

 \bigskip 
\noindent {\it 8.1 FGTF nilpotent groups with rank one center}

\noindent
Let $\Gamma$ be a FGTF nilpotent group and let 
$Z(\Gamma)$ be its center.

\begin{lemma}\label{TransCond} Let's asssume that $\Gamma$ is 
self-similar and
$Z(\Gamma)\simeq \Z$. Then $\Gamma$ is transitive self-similar. 
\end{lemma}

\begin{proof} Assume that $\Gamma$ admits a faithful
self-similar action on some $A^\omega$, where   $A$ is a finite 
alphabet. Let $a_1,\dots,a_k$ be a set of representatives of 
$A/\Gamma$, where  $k$ is the number of $\Gamma$-orbits on $A$. For each $1\leq i\leq k$, let $\Gamma_i$ be
the stabilizer of $a_i$. For any $h\in\Gamma_i$, there is
$h_i\in\Gamma$ such that

\centerline{$h(a_i w)=a_i h_i(w)$,}

\noindent for all $w\in A^\omega$.  Since the action is faithfull
$h_i$ is uniquely determined and the map 
$f_i:\Gamma_i\to\Gamma, h\mapsto h_i$ 
 is a group morphism.

Let $\fn^\Q$ be the Malcev Lie algebra of $\Gamma$, and let $\fz$ be its center, and
let $z\neq 0$ be a generator of $\cap_i\,Z(\Gamma_i)$. 
 By Lemma \ref{extension}, the group morphism 
$f_i$ extends to a Lie  algebra morphism $\tilde f_i:\fn^\Q\to\fn^\Q$. Since 
$\fz=\Q\otimes Z(\Gamma)$ is one dimensional, it follows that
 either $\tilde f_i$ is an isomorphism or
 $\tilde f_i(\fz)=0$. In any case, 
we have $\tilde f_i(z)=x_i z$, for some
$x_i\in\Q$. However $\Z z$ is not invariant by all $\tilde f_i$, 
otherwise it would be in the kernel of the action.
It follows that
at least one $x_i$ is not an integer. 

For such an index $i$, the  $f_i$-core of  $\Gamma_i$ is trivial, and the virtual
morphism  $(\Gamma_i,f_i)$  is a self-similar datum for $\Gamma$.
Thus $\Gamma$ is transitive self-similar.
 
\end{proof}

\bigskip 
\noindent {\it 8.2  Small representations}

\noindent
Let $N$ be a CSC nilpotent  Lie group with Lie algebra
$\fn^\R$  and let $\Gamma$ be a cocompact lattice.

\begin{lemma}\label{small} If $\Gamma$ is transitive self-similar, then
 there exists a faithfull $\fn^\R$-module of dimension
$1+\dim \fn^\R$.
\end{lemma}

\begin{proof}
By hypothesis, $\Gamma$ is transitive self-similar.
By Theorem 2, $\fz^\C$ admits a special grading

\centerline{$\fn^\C=\oplus_{n\in\Z}\,\fn^{\C}_{n}$.}

\noindent Let $\delta: \fn^\C\to\fn^\C$ be the derivation 
defined by $\delta (x) = nx$ if $x\in\fn_n$. Since
$\delta\vert\fz^\C$ is injective, it follows
that there is some $\partial\in\Der\,\fn^\R$ such that
$\partial\vert\fz^\R$ is injective.

Set $\fm^\R=\R\partial\ltimes\fn^\R$. Relative to the adjoint action, 
$\fm^\R$ is a faithfull $\fz^\R$-module.
Therefore $\fm^\R$ is a  faithfull $\fn^\R$-module
with the prescribed dimension.
\end{proof}

\bigskip 
\noindent {\it 8.3  Filiform nilpotent Lie algebras}

\noindent Let $\fn$ be a nilpotent Lie algebra over $\Q$. Let
$C^n\fn$ be the decreasing central series, which is inductively
defined by
$C^1\fn=\fn$ and $C^{n+1}\fn=[\fn,C^n\fn]$. The nilpotent Lie algebra
$\fn$ is called {\it filiform} if $\dim C^1\fn/C^2\fn=2$ and
$\dim C^k\fn/C^{k+1}\fn\leq 1$ for any $k>1$. Set
$n=\dim \fn$. It follows from the definition that
$\dim C^k\fn/C^{k+1}\fn=1$ for any $0<k\leq n-1$ and
$C^k\fn=0$ for any $k\geq n$.

\begin{lemma}\label{fil} Let $\fn$ be a filiform nilpotent Lie algebra over $\Q$,
with $\dim \fn\geq 3$.
Then its center $\fz$ has dimension one.
\end{lemma}

\begin{proof} Let $z\in \fn$ be nonzero. Let
 $k$ be the integer such that $z\in C^k\fn\setminus C^{k+1}\fn$.
Since $C^{k}\fn=C^{k+1}\fn\oplus\Q z$

\centerline{$C^{k+1}\fn=[\fn, C^k\fn]=[\fn, C^{k+1}\fn]
+[\fn,z]= C^{k+2}\fn$.}

 \noindent It follows that $C^{k+1}\fn=0$. Therefore
 $\fz$ lies in $C^k\fn$, which is a one dimensional ideal.
  \end{proof}

 \bigskip 
\noindent {\it 8.4  Benoist Theorem}
 
\begin{Benoistthm} There is a  nilpotent Lie algebra $\fn_B^\R$
  of dimension $11$ over $\R$, with the following properties
 
 (i) The Lie algebra $\fn_B^\R$ has no faithfull representations of
 dimension $12$,
 
 (ii) the Lie algebra $\fn_B^\R$ is defined over $\Q$, and
 
 (iii) the Lie algebra $\fn_B^\R$ is filiform.
 
\end{Benoistthm}
 
 The three assertions appear in
different places of \cite{Ben}. Indeed
 Assertion (i), which is explicitely stated in Theorem 2 of \cite{Ben},   hold for a one-parameter family of eleven dimensional
 Lie algebras, which are denoted $\fa_{-2,1,t}$ in section 2.1 of \cite {Ben}. These Lie algebras are filiform by Lemma 4.2.2 of \cite{Ben}. Moreover, when $t$ is rational, $\fa_{-2,1,t}$ is defined  over $\Q$. Therefore the Benoist Theorem holds for the Lie algebras 
 $\fn_B=\fa_{-2,1,t}$ where $t$ is any rational number.

 \bigskip 
\noindent {\it 8.5  A FGTF group which is not self-similar}

\noindent
Let $N_B$ the CSC nilpotent Lie group with
Lie algebra $\fn_B^\R$. Since $\fn_B^\R$ is defined over 
$\Q$, $N_B$ contains some cocompact lattice.

\begin{cor}\label{Nonaffine} Let $\Gamma$ be any cocompact   lattice in $N_B$. Then $\Gamma$ is not self-similar.
\end{cor}

\begin{proof} Let's assume otherwise. 
By Benoist Theorem and Lemma \ref{fil},
the center of
$\fn_B^\R$ is one dimensional. 
Thus the center of $\Gamma$ has rank one,
and by Lemma \ref{TransCond}, $\Gamma$ is transitive self-similar. By Lemma \ref{small},
$\fn_B^{\R}$ admits a faithfull representation of dimension 12, which
contradicts Benoist Theorem. 

Therefore $\Gamma$ is not self-similar.
\end{proof}

 \bigskip 
\noindent {\it 8.6  On the Scheuneman-Milnor conjecture}

\noindent
A smooth manifold $M$ is called {\it affine} if it admits a 
torsion-free and flat connection. Scheuneman \cite{Sch} and Milnor \cite{Mi} asked the following question

\centerline{\it is any nilmanifold $M$ affine?}

\noindent The story of the Scheuneman-Milnor conjecture is  quite interesting.
For many years, there are been a succession of proofs
followed by refutations, but there was no doubts that the
conjecture should be ultimalely proved... 
until a counterexample has been found by Benoist \cite{Ben}.
Indeed it is an easy   corollary of his
previously mentionned Theorem. 

The  following question is a refinement of the previous conjecture

\centerline{\it  if $\pi_1(M)$ is densely self-similar,
is the nilmanifold $M$ affine?} 

\noindent A positive result in that direction is

\begin{cor}\label{Affine} Let $M$ be a nilmanifold. If
$\pi_1(M)$ is freely  self-similar,
then $M$ is affine complete.
\end{cor}

\begin{proof} Set $M=N/\Gamma$, where $N$ is a CSC nilpotent
Lie group and $\Gamma$ is a cocompact lattice. Let
$\fn^\R$ be the Lie algebra of $N$. By Theorem \ref{Main2},
$\C\otimes\fn^\R$
admits a very special grading, what implies that a generic
derivation is injective. Therefore there is a derivation
$\delta$ of $\fn^\R$ which is injective. Set
$\fm^\R=\R\delta\ltimes \fn^R$. Then $N$ is equivariantly
diffeomorphic to the affine space
$\delta+\fn^\R\subset \fm^\R$. Therefore $M$ is affine complete.

\end{proof}

 \section{Absolute Complexities}
 
 For the whole chapter,  $N$ will be a CSC nilpotent Lie groups, with Lie algebra $\fn^\R$. Let's assumethat that $N$ contains some  cocompact lattices.  
 
 Under the condition of  Theorem \ref{Main1} or \ref{Main2}, 
 any cocompact lattice  $\Gamma$ in $N$ admits a transitive
 or free self-similar action on some  $A^\omega$. In this section, we try to determine the minimal degree of these actions.
 
\bigskip
 \noindent{\it 9.1 Three type of absolute complexities}
 
 \noindent The {\it complexity} of a  cocompact lattice $\Gamma\subset N$, denoted by $\cp\,\Gamma$, is the smallest
 degree of a faithfull transitive self-similar action of $\Gamma$
 on some $A^\omega$, with the convention that $\cp\,\Gamma=\infty$ if $\Gamma$ is not transitive self-similar. Similarly,
 the {\it free complexity} of  $\Gamma$, 
 denoted by $\fcp\,\Gamma$, is the smallest
 degree of a free self-similar action of $\Gamma$.
 Two cocompact lattices are called {\it commensurable} if
 they share a commoun subgroup of finite index.
 The complexity and the free complexity of a commensurable class 
 $\xi$ are the integers
 
 \centerline{$\cp\,\xi=\Min_{\Gamma\in\xi}\,\cp\,\Gamma$, and}
 
 \centerline
 {$\fcp\,\xi=\Min_{\Gamma\in\xi}\,\fcp\,\Gamma$.}
 
 Then, the complexity of the nilpotent group $N$ is 
 
 \centerline{$\cp\,N=\Max_{\xi}\,\cp\,\xi$,}
  
  \noindent where $\xi$ runs over all commensurable classes in $N$.
  In what follows, we will provide a formula for
the complexity of commensurable classes. The question

\centerline{\it under which condition $\cp N<\infty$?}

\noindent is not solved, but it is a deep question. In  chapter 10, 
a class of CSC nilpotent Lie groups of infinite complexity is investigated.

\bigskip
\noindent {\it 9.2 Theorem \ref{Main3}}

\noindent Let $\xi$ be a commensurable class of cocompact lattices
in $N$, and let $\Gamma\in\xi$. The Malcev Lie
algebra $\Gamma$ is a $\Q$-form of the Lie algebra
$\fn^\R$. Since it depends only on $\xi$, it will
be denoted by $\fn(\xi)$.

\begin{thm}\label{Main3} We have

\centerline{
$\cp\,\xi=\Min_{h\in{\cal S}(\fn(\xi))}\,\he(h)$, and}

\centerline{
$\fcp\,\xi=\Min_{h\in{\cal V}(\fn(\xi))}\,\he(h)$.}

\end{thm} 

\begin{proof} Let $h\in{\cal S}(\fn(\xi))$ be an isomorphism
of minimal height. In order to show that
$\cp\,\xi=\he(h)$, we can  assume that $h$ is semi-simple, 
by Lemma \ref{Chevalley}.

Further, let $\Gamma$ be any cocompact lattice in $\xi$.
By Lemma  \ref{ssdatum}, we have 
$\cp\,\Gamma=\Min_{f\in {\cal S}(\fn(\xi))}\,\cp_f\,\Gamma$. 
By lemma \ref{multineq}, we have $\cp_f\,\Gamma\geq\he(f)$, therefore
we have $\cp\,\Gamma\geq \he(h)$. In particular
$\cp\,\xi \geq \he(h)$.

By Lemma \ref{mingood}, $\Gamma$ contains a finite index 
subgroup $\tilde\Gamma$
which is minimal relative to $h$. 
Since $\cp_h\,\tilde\Gamma=\he(h)$, it follows that
$\cp\,\xi\leq\he(h)$.

Therefore $\cp\,\xi=\he(h)$ and the first assertion  is proved.

For the second assertion, let's notice that an free action
of minimal degree is automatically transitive, see the proof
of Theorem \ref{Main2}. Then the rest of the proof is strictly
identical to the previous proof.
\end{proof}

\bigskip\noindent 
 {\it 9.3 Classification of lattices in a CSC nilpotent Lie groups}

\noindent  Obviously  Malcev's Theorem 
implies the following

\begin{Malcevcor} The map $\xi\mapsto \fn(\xi)$
establishes a bijection between the commensurable classes of 
lattices and the $\Q$-forms of the Lie algebra $\fn^{\R}$.
\end{Malcevcor} 

For the next chapter, it is interesting to translate this
into the framework of non-abelian Galois cohomology.
Somehow, it is more concrete, since the non-abelian Galois cohomology
classifies $\Q$-forms of classical objects. 

Set ${\bf G}=\Aut \fn^\C$, let ${\bf U}$ be its unipotent 
radical and set $\overline{\bf G}={\bf G}/{\bf U}$. 
From now on, fix once for all a commensurable class $\xi_0$ 
of cocompact lattices. Then
$\fn(\xi_0)$ is a $\Q$-form of $\fn^{\C}$, what provides
a $\Q$-form of the algebraic groups ${\bf G}$ and $\overline{\bf G}$. It induces
an action of $\Gal(\Q)$ over 
$\overline{\bf G}(\overline\Q)$. 

Set ${\overline \Q}_{re}=\overline \Q\cap \R$ and let

\centerline{$\overline{\pi}:
H^1(\Gal(\Q), {\bf \overline G}(\overline{\Q}))
\rightarrow 
H^1(\Gal({\overline \Q}_{re}), {\bf \overline G}(\overline{\Q}))$}

\noindent be the natural map. Recall that these 
two non-abelian cohomologies
are pointed sets, where the distinguished point $*$ comes from
the given $\Q$-form and the induced 
${\overline \Q}_{re}$-form. Denote by $\Ker\,\overline{\pi}$ the kernel
of $\overline{\pi}$, i.e. the fiber ${\overline\pi}^{-1}(*)$ of the distinguished point.

Let ${\cal L}(N)$ be the set of all commensurable classes classes of lattices of $N$, up to conjugacy.

\begin{cor}\label{classification}
There is a natural identification

\centerline{${\cal L}(N)\simeq\, \Ker\,\overline{\pi}$.}
\end{cor}

\begin{proof} 
For any field $K\subset \C$,  set
$\fn^K=K\otimes\fn(\xi_0)$. For any two fields $K\subset L\subset \C$,
let  ${\cal F}(L/K)$ be the set
of $K$-forms of $\fn^L$, up to conjugacy. Then 
${\cal F}(L/K)$ is a pointed set, whose distinguished point is the
$K$-form $\fn^K$.

By the Lefschetz principle, the 
$\Q$-forms of $\fn^\C$ (up to conjugacy) are in 
bijection with the $\Q$-forms of $\fn^{\overline{\Q}}$.
Similarly by the Tarski-Seidenberg principle the real forms (up to conjugacy) of 
$\fn^\C$ are in bijection with the
${\overline \Q}_{re}$-forms of 
$\fn^{\overline \Q}$. So
we have

\centerline{${\cal F}(\C/\Q)\simeq {\cal F}(\overline\Q/\Q)$
and ${\cal F}(\C/\R)\simeq {\cal F}(\overline\Q/{\overline \Q}_{re})$}

Since a Lie algebra is a vector space 
endowed with a tensor (its Lie bracket), it follows from
\cite{Ser}, III-2, Proposition 1 that

\centerline{ ${\cal F}(\overline\Q/\Q)=
H^1(\Gal(\Q), {\bf  G}(\overline{\Q}))$,  and}

\centerline{
${\cal F}(\overline\Q/{\overline \Q}_{re})=
H^1(\Gal({\overline \Q}_{re}), {\bf  G}(\overline{\Q}))$.}

\noindent Moreover since ${\bf U}$ is unipotent, we have

\centerline{$H^1(\Gal(\Q), {\bf  G}(\overline{\Q}))
\simeq H^1(\Gal(\Q), {\bf \overline G}(\overline{\Q}))$,  
and}

\centerline{
$H^1(\Gal(\Q_{re}), {\bf  G}(\overline{\Q}))
\simeq H^1(\Gal(\Q_{re}), {\bf \overline G}(\overline{\Q}))$.}

There is a commutative diagram of pointed sets

 \begin{tabular}{c c c}
 
 ${\cal F}(\C/\Q)$ & 
 $\overset{\theta}{\longrightarrow}$
 & ${\cal F}(\C/\R)$\\
 
 $\downarrow$ &
 
 & $\downarrow$ \\
 
 ${\cal F}(\overline\Q/\Q)$ &
 $\overset{\theta'}{\longrightarrow}$
& ${\cal F}(\overline\Q/{\overline \Q}_{re})$ \\

$\downarrow$ && $\downarrow$ \\ 
 
 $H^1(\Gal(\Q), {\bf  G}(\overline{\Q}))$ &
 $\overset{\pi}{\longrightarrow}$ 
& $ H^1(\Gal({\overline \Q}_{re}), {\bf  G}(\overline{\Q}))$\\
 
 $\downarrow$ && $\downarrow$ \\
 
 $ H^1(\Gal(\Q), {\bf\overline  G}(\overline{\Q}))$ &
 $\overset{\overline\pi}{\longrightarrow}$ &
 $ H^1(\Gal({\overline \Q}_{re}), {\bf\overline  G}(\overline{\Q}))$\\

 \end{tabular}
 
\noindent where $\theta$ is the map $\R\otimes_{\Q}\,-$ ,
 $\theta'$ is the map ${\overline\Q}_{re}\otimes_{\Q}\,-$,
 $\pi$ and $\overline\pi$ are restrictions maps.
 It is tautological that
${\cal L}(N)=\Ker\, \theta$. Since all vertical maps are 
bijective, it follows that
${\cal L}(N)$ is isomorphic to 
$\Ker\,\overline{\pi}$.

\end{proof}

\section{ Some Nilpotent Lie groups of infinite complexity.}
 
This chapter is devoted to the analysis to a
class of  CSC nilpotent Lie groups ${\cal N}$,
for which the classification of commensurable classes
and the computation of their complexity are very explicitely
connected with the arithmetic of complex quadratic fields.

For $K=\R$ or $\C$, let $O(2,K)$ be 
the group of linear automorphisms
of $\R^2$ preserving the quadratic form $x^2+y^2$.
Let ${\cal L}$ be the class of nilpotent Lie algebras
$\fn^R$  over $\R$ satisfying the following properties

(i) $\fn^R$  has a $\Q$-form

(ii) $\fn^R/[\fn^R,\fn^R]\simeq \R^2$ has dimension two

(iii) the Lie algebra $\fn^\C:=\C\otimes\fn^\R$ has 
a special grading

(iv) for $K=\R$ or $\C$, the image of $\Aut\,\fn^K$ in 
$GL(\fn^K/[\fn^K,\fn^K])$ is $O(2,K)$.

\noindent Let  be the class of CSC
nilpotent Lie groups $N$ whose  Lie algebra $\fn^\R$ is in
${\cal L}$.

It should be noted that the class ${\cal N}$ is {\it not} empty.
There is one Lie group $N_{112}\in{\cal N}$  of dimension $112$,
see \cite{Mat}. Indeed \cite{Mat} contains a general method
to find nilpotent Lie algebras with a prescribed group
of automorphisms, modulo its unipotent radical.
For the group $O(2,\R)$, $N_{112}$ is the Lie group of minimal dimension obtained with this method. However it is difficult to provide more details, without going to very long explanations. 

From now on, $N$ will be any Lie group in class ${\cal N}$,
$\xi_0$ will be one commensurable class of lattices in $N$
and $\fn:=\fn(\xi_0)$ will be the corresponding  corresponding $\Q$ form of $\fn^R$. As before, set $\fn^K=K\otimes\fn$ for any
field $K\subset\C$. Let ${\bf G}={\bf Aut}\,\fn$ the algebraic automorphism group of $\fn$, let ${\bf U}$ be its unipotent radical
and set ${\bf \overline G}={\bf G}/{\bf U}$.
By hypothesis,  ${\bf \overline G}$ is the algebraic group
$O(2)$.

\bigskip
\noindent
{\it 10.1 The $\Z$-grading of $\fn^\C$}

\noindent
Since ${\bf \overline G}(\C)=O(2,\C)$, a maximal torus 
${\bf H}$ of ${\bf G}(\C)$ has dimension $1$. Therefore $\fn^\C$ has a  $\Z$-grading 

\centerline{$\fn^\C=\oplus_{k\in\Z}\, \fn^\C_k$,}

\noindent satisfying the following properties

(i) the grading is essentially unique, namely any other
grading is a multiple of the given grading,

(ii) $\dim \fn^\C_k=\dim \fn^\C_{-k}$ for any $k$.
In particular $\fn^\C$ does {\it not} admit a (non-trivial)
non-negative grading, and

(iii) the grading is {\it not} defined over $\R$.

\noindent Indeed since  ${\bf \overline G}(\C)=O(2,\C)$, 
the normalizer  ${\bf K}(\C)$  of ${\bf H}(\C)$ has two connected components, and any $\sigma\in {\bf K}(\C)\setminus {\bf K}(\C)^0$
exchanges $\fn^\C_k$ and $\fn^\C_{-k}$, what shows Assertion (ii).
Since ${\bf \overline G}(\R)=O(2,\R)$, no torus of 
${\bf G}(\R)$ is split, what implies Assertion (iii).

Moreover, the grading is not very special, so $\fcp(\xi)=\infty$
for any commensurable class $\xi$.
For the forthcoming  computation of $\cp(\xi)$, the following quantity will be involved

\centerline{ $e(N)=\sum_{k> 0} \,k\dim\fn^\C_k$.}

\noindent For example, for the Lie group $N_{112}$ of
\cite{Mat}, we have  $e(N_{112})=126$.

\bigskip
\noindent
{\it 10.2 Classification of commensurable lattices in $N$}

\begin{lemma} Let $N\in{\cal N}$. Up to conjugacy, there is a bijection between

(i) the commensurable class of cocompact lattices in $N$, and

(ii) the positive definite quadratic form on $\Q^2$.
\end{lemma}

\begin{proof} Let $q_0$ be a given definite quadratic form on 
$\Q^2$. It determines a $\Q$-form of the algebraic group $O(2)$,
and $H^1(\Gal(\Q),O(2,\overline\Q))$ classifies the
quadratic forms on $\Q^2$, while the kernel of

\centerline{$H^1(\Gal(\Q),O(2,\overline\Q))
\to H^1(\Gal(\overline \Q_{re}),O(2,\overline\Q))$}

\noindent classifies the positive definite quadratic forms
 on $\Q^2$. Thus the lemma follows from
Corollary \ref{classification}.

\end{proof}

The classification of positive definite
quadratic forms $q$ on $\Q^2$ is well
known. Up to conjugacy, $q$ can be written as

\centerline{$q(x,y)=ax^2+ ad y^2$,}

\noindent where $a, \,d$ are positive and $d$ is a square-free integer. Then $q$ is determined
by the following two invariants

(i) its discriminant $-d$, viewed as an element of $\Q^*/\Q^{*2}$,

(ii) the value $a$, viewed as an element in $\Q^*/ N_{K/\Q}(K^*)$, where $K=\Q(\sqrt{-d})$. Equivalently, this means that
$q(\Q^2\setminus 0)=a N_{K/\Q}(K^*)$.

For any positive definite quadratic forms $q$ on $\Q^2$, 
let $\xi(q)$ be the corresponding commensurable class
(or more precisely, the conjugacy class of the commensurable class).
By Theorem \ref{Main3}, $\cp\,\xi(q)$ only depends on $O(q)$, therefore
it only depends on the discriminant $-d$.

\bigskip
\noindent
{\it 10.3 The function $F(d)$}

\noindent
 Let $d$ be a positive square-free integer. Set
$K=\Q(\sqrt{-d})$, let ${\cal O}$ be its ring of integers,
let $R$ be te set of roots of unity in $K$ and
set $K_1=\{z\in K\vert z\overline z=1\}$. 
For $z\in K^*$, recall that the integer $d(z)$ is defined by
$d(z)=N_{K/\Q}(\pi_z)=\Card {\cal O}/\pi_z$,
where $\pi_z$ is the ideal 
$\pi_z=\{a\in {\cal  O}\vert az\in{\cal O}\}$.
Set

\centerline{$F(d)=\Min_{z\in K_1\setminus R}\,d(z)$.}

\noindent We will now show two formulas for $F(d)$.
Indeed $F(d)$ is the norm of some specific ideal
in $K=\Q(\sqrt{-d})$, and it is also the minimal
solution of some diohantine equation.

Let ${\cal J}$ be the set of all ideals $\pi$ of ${\cal O}$
such that $\pi$ and $\overline \pi$ are coprime and
$\pi^2$ is principal.

\begin{lemma} \label{F(d)1} We have 

\centerline{$F(d)=\Min_{\pi\in{\cal J}}\,N_{k/\Q}(\pi)$.}

In particular, we have $F(1)=5$ and $F(3)=7$.
\end{lemma}

\begin{proof}
The map $z\mapsto \pi_z$ induces a bijection
$(K_1\setminus R)/R\simeq {\cal J}$, from which the first assertion follows. Moreover, if $\Cl(K)=\{0\}$, then $F(d)$ is the smallest
split prime number. Therefore $F(1)=5$ and $F(3)=7$.
\end{proof}

Let's consider the following diophantine equation

$({\cal E})$ 
\hskip1.5cm$4n^2=a^2+db^2$, with $n>0$, $a> 0$ and $b\neq 0$.

\noindent A solution $(n, a,b)$ of $({\cal E})$  is
called {\it primitive} if $\gcd (n,a)=1$. Let 
$\Sol({\cal E})$ (respectively $\Sol_{prim}({\cal E})$)
be the set of  solutions (respectively of primitive solutions)
of $({\cal E})$.

Let $\pi\in {\cal J}$. Since $\pi^2$ is principal, there are
integers  $a(\pi)>0$ and $b(\pi)$ such that
$a(\pi)+b(\pi)\sqrt{-d}$ is a generator of $4\pi^2$.
Moreover, let's assume that $d\neq 1$ or $3$. Then
$R=\{\pm 1\}$ and the integers $a(\pi)$ and $b(\pi)$
are uniquely determined. Thus there is a map
$\theta: {\cal J}\to \Sol({\cal E})$ defined by

\centerline{
$\theta(\pi)=(N_{K/\Q}(\pi),a(\pi), b(\pi))$.}

\begin{lemma}\label{F(d)2} Under the hypothesis that $d\neq 1$ or $3$,
the map $\theta$ induces a bijection from
${\cal J}$ to $\Sol_{prim}({\cal E})$. In particular

\centerline{$F(d)=\Min_{(n,a,b)\in \Sol({\cal E})}\,n$.}
\end{lemma}

\begin{proof}

{\it Step 1: proof that 
$\theta({\cal J})\subset\Sol_{prim}({\cal E})$.}  An algebraic integer $z\in{\cal O}$ is called 
primitive if there are no integer $d>1$ such that
$z/d$ is an algebraic integer. Equivalently, there are 
no integer $d>1$ such that $d\mid z+\overline z$ and 
$d^2\mid z\overline z$.

Let $\pi\in {\cal J}$ and set $z=1/2(a(\pi)+b(\pi)\sqrt{-d})$.
Since $z+{\overline z}= a(\pi)$ and 
$z.{\overline z}=N_{K/\Q}(\pi^2)$, 
$z$ is an algebraic integer
which is a generator of $\pi^2$. Since $\pi^2$
and $\overline{\pi}^2$ are coprime, $z$ is primitive.
Since $z.{\overline z}=
N_{K/\Q}(\pi)^2$, it follows 
that $N_{K/\Q}(\pi)$ and $a(\pi)$ are coprime.
Hence $\theta(\pi)\in \Sol_{prim}({\cal E})$ and the claim is proved.

\noindent
{\it Step 2: proof that 
$\theta({\cal J})=\Sol_{prim}({\cal E})$.} Let
$(n,a,b)\in \Sol_{prim}({\cal E})$ and $z=1/2(a(\pi)+b(\pi)\sqrt{-d})$.
Since $z\neq \overline z$, $z+{\overline z}=a$ and
$z{\overline z}=n$, the  number $z$ is an algebraic integer.
Set $\tau=z{\cal O}$ and let

\centerline{$\tau=\pi_1^{m_1}\dots \pi_k^{m_k}$}

\noindent be the factorization of $\tau$ into a product of prime ideals of ${\cal O}$, where, as usual we assume that 
$\pi_i\neq \pi_j$ for $i\neq j$ and all $m_i$ are positive. 

For  $1\leq i\leq k$,  let $p_i$ be the characteristic of the field
${\cal O}/\pi_i$. Since $n$ and $a$ are
coprime, $\tau$ and $\overline{\tau}$ are coprime. It
follows that $\overline{\pi_i}$ does not divide $\tau$.
In particular $\pi_i\neq \overline{\pi_i}$ and
$N_{K/\Q}(\pi_i)=p_i$. Since $\pi_i$ and $\overline{\pi_i}$
are the only two ideals over $p_i$, we have
$m_i=v_{p_i} (N_{K/\Q}(\tau))=v_{p_i}(n^2)$. Since each
$m_i$ is even, we have $\tau=\pi^2$ for some ideal
$\pi\in{\cal J}$. Therefore $\theta(\pi)=(n,a,b)$,
and the claim is proved.

\noindent {\it Step 3.} It follows easily that $\theta$ is
a bijection  from
${\cal J}$ to $\Sol_{prim}({\cal E})$. In particular
$F(d)=\Min_{(n,a,b)\in \Sol_{prim}({\cal E})}\,n$, from which the lemma follows.

\noindent

\end{proof}

\bigskip
\noindent
{\it 10.4 Complexity computation}

\begin{thm} \label{Main4} Let $q$ be a positive definite quadratic form
on $\Q^2$ of discriminant $-d$. Then we have

\centerline{$\cp\,\xi(q)= F(d)^{e(\fn^\C)}$}
\end{thm}

\begin{proof} {\it Step 1.} 
Let $G\subset \End_\Q(K)$ be the group generated by
the multiplication by elements in $K_1$ and by the complex conjugation.  We have $G\simeq O(2)$ and 
$SO(2)\simeq K_1$. As a $O(q)$-module,
there is an isomorphism

\centerline{$V\simeq \Q(\sqrt{-d})$,}

\noindent where $V=\fn(\xi(q))/[\fn(\xi(q)),\fn(\xi(q))]$. 

\noindent {\it Step 2.} Let  ${\overline{\cal S}}(\fn(\xi(q))$ be 
the image of ${\cal S}(\fn(\xi(q))$ in $O(q)$. We claim
that

\centerline{${\overline{\cal S}}(\fn(\xi(q))=K_1\setminus R$.}

\noindent Indeed $O(q)$ can be identified with a Levi factor of
${\bf G}(\Q)$ and let $\rho:O(q)\to{\bf G}(\Q)$ a corresponding lift. Any element in $R\cup O(q)\setminus SO(q)$
has finite order, hence we have

\centerline{${\overline{\cal S}}(\fn(\xi(q))
\subset K_1\setminus R$.}

\noindent Let $z\in K_1\setminus R$.  It is clear that $z$ is
not an algebraic integer. Since the grading is special,
we have 

\centerline{$\fz^{\C}=\oplus_{k\neq 0} \fz^{\C}_k$.}

 \noindent Since the eigenvalues of  $\rho(z)$ on $\fz_k$ is $z^k$, it follows that $z$ belongs to
${\overline{\cal S}}(\fn(\xi(q))$, what proves the point.
 
\noindent {\it Step 3.} Let $z\in  K_1\setminus R$.
We have $\overline z=\overline z^{-1}$. Therefore by
Lemma \ref{formula} we have

\centerline {$\he\,\rho(z)=\prod_{k\geq 1} d(z^k)^{\dim\,\fn^\C_k}
=d(z)^{e(N)}$.}

\noindent Therefore Theorem 4 implies Theorem 5.
\end{proof}

Since $F(d)\geq {\sqrt{1+d}\over 2}$, it follows that

\begin{cor}\label{Ncomplexity} The group $N$ has infinite complexity.
\end{cor}

Since $F(7)=F(15)=2$ and $F(d)\geq 3$ otherwise, it follows that

\begin{cor}\label{mincp} If the positive definite quadratic form 
$q$ has  discriminant $-7$ or $-15$ we have

\centerline{$\cp\,\xi(q)= 2^{e(N)}$,}

\noindent and $\cp\,\xi(q)\geq 3^{e(N)}$ otherwise.
\end{cor}

\end{document}